\documentclass[a4paper,twoside]{amsart}

\usepackage{amsmath,amsfonts,amsthm,mathrsfs}
\usepackage{amssymb}

\usepackage[unicode,bookmarks]{hyperref}
\usepackage[alphabetic,initials]{amsrefs}

\usepackage{ifthen}
\usepackage{enumitem}

\usepackage{tikz}
\usetikzlibrary{matrix,arrows}
\newlength{\myarrowsize} 
\pgfarrowsdeclare{cmto}{cmto}{
	\pgfsetdash{}{0pt} 
	\pgfsetbeveljoin 
	\pgfsetroundcap 
	\setlength{\myarrowsize}{0.6pt}
	\addtolength{\myarrowsize}{.5\pgflinewidth}
	\pgfarrowsleftextend{-4\myarrowsize-.5\pgflinewidth} 
	\pgfarrowsrightextend{.8\pgflinewidth}
}{
	\setlength{\myarrowsize}{0.6pt} 
  	\addtolength{\myarrowsize}{.5\pgflinewidth}  
	\pgfsetlinewidth{0.5\pgflinewidth}
	\pgfsetroundjoin
	\pgfpathmoveto{\pgfpoint{1.5\pgflinewidth}{0}}
	\pgfpatharc{-109}{-170}{4\myarrowsize}
	\pgfpatharc{10}{189}{0.58\pgflinewidth and 0.2\pgflinewidth}
	\pgfpatharc{-170}{-115}{4\myarrowsize+\pgflinewidth}
	\pgfpathclose
	\pgfusepathqfillstroke
	\pgfpathmoveto{\pgfpoint{1.5\pgflinewidth}{0}}
	\pgfpatharc{109}{170}{4\myarrowsize}
	\pgfpatharc{-10}{-189}{0.58\pgflinewidth and 0.2\pgflinewidth}
	\pgfpatharc{170}{115}{4\myarrowsize+\pgflinewidth}
	\pgfpathclose
	\pgfusepathqfillstroke
	\pgfsetlinewidth{2\pgflinewidth}
}

\pgfarrowsdeclare{cmonto}{cmonto}{
	\pgfsetdash{}{0pt} 
	\pgfsetbeveljoin 
	\pgfsetroundcap 
	\setlength{\myarrowsize}{0.6pt}
	\addtolength{\myarrowsize}{.5\pgflinewidth}
	\pgfarrowsleftextend{-4\myarrowsize-.5\pgflinewidth} 
	\pgfarrowsrightextend{.8\pgflinewidth}
}{
	\setlength{\myarrowsize}{0.6pt} 
  	\addtolength{\myarrowsize}{.5\pgflinewidth}  
	\pgfsetlinewidth{0.5\pgflinewidth}
	\pgfsetroundjoin
	\pgfpathmoveto{\pgfpoint{1.5\pgflinewidth}{0}}
	\pgfpatharc{-109}{-170}{4\myarrowsize}
	\pgfpatharc{10}{189}{0.58\pgflinewidth and 0.2\pgflinewidth}
	\pgfpatharc{-170}{-115}{4\myarrowsize+\pgflinewidth}
	\pgfpathclose
	\pgfusepathqfillstroke
	\pgfpathmoveto{\pgfpoint{1.5\pgflinewidth}{0}}
	\pgfpatharc{109}{170}{4\myarrowsize}
	\pgfpatharc{-10}{-189}{0.58\pgflinewidth and 0.2\pgflinewidth}
	\pgfpatharc{170}{115}{4\myarrowsize+\pgflinewidth}
	\pgfpathclose
	\pgfusepathqfillstroke
	\pgfpathmoveto{\pgfpoint{1.5\pgflinewidth-0.3em}{0}}
	\pgfpatharc{-109}{-170}{4\myarrowsize}
	\pgfpatharc{10}{189}{0.58\pgflinewidth and 0.2\pgflinewidth}
	\pgfpatharc{-170}{-115}{4\myarrowsize+\pgflinewidth}
	\pgfpathclose
	\pgfusepathqfillstroke
	\pgfpathmoveto{\pgfpoint{1.5\pgflinewidth-0.3em}{0}}
	\pgfpatharc{109}{170}{4\myarrowsize}
	\pgfpatharc{-10}{-189}{0.58\pgflinewidth and 0.2\pgflinewidth}
	\pgfpatharc{170}{115}{4\myarrowsize+\pgflinewidth}
	\pgfpathclose
	\pgfusepathqfillstroke
	\pgfsetlinewidth{2\pgflinewidth}
}

\pgfarrowsdeclare{cmhook}{cmhook}{
	\pgfsetdash{}{0pt} 
	\pgfsetbeveljoin 
	\pgfsetroundcap 
	\setlength{\myarrowsize}{0.6pt}
	\addtolength{\myarrowsize}{.5\pgflinewidth}
	\pgfarrowsleftextend{-4\myarrowsize-.5\pgflinewidth} 
	\pgfarrowsrightextend{.8\pgflinewidth}
}{
	\setlength{\myarrowsize}{0.6pt} 
  	\addtolength{\myarrowsize}{.5\pgflinewidth}  
 	\pgfsetdash{}{0pt}
	\pgfsetroundcap
	\pgfpathmoveto{\pgfqpoint{0pt}{-4.667\pgflinewidth}}
	\pgfpathcurveto
    {\pgfqpoint{4\pgflinewidth}{-4.667\pgflinewidth}}
    {\pgfqpoint{4\pgflinewidth}{0pt}}
    {\pgfpointorigin}
	\pgfusepathqstroke
}

\newenvironment{diagram}[2]{%
\begin{equation}%
\begin{tikzpicture}[>=cmto,baseline=(current bounding box.center),%
	to/.style={->,font=\scriptsize,cap=round},%
	into/.style={cmhook->,font=\scriptsize,cap=round},%
	onto/.style={-cmonto,font=\scriptsize,cap=round},%
	math/.style={matrix of math nodes, row sep=#2, column sep=#1,%
		text height=1.5ex, text depth=0.25ex}]%
}{%
\end{tikzpicture}%
\end{equation}% 
\ignorespacesafterend%
}
\newenvironment{diagram*}[2]{%
\[%
\begin{tikzpicture}[>=cmto,baseline=(current bounding box.center),%
	to/.style={->,font=\scriptsize,cap=round},%
	into/.style={cmhook->,font=\scriptsize,cap=round},%
	onto/.style={-cmonto,font=\scriptsize,cap=round},%
	math/.style={matrix of math nodes, row sep=#2, column sep=#1,%
		text height=1.5ex, text depth=0.25ex}]%
}{%
\end{tikzpicture}%
\]% 
\ignorespacesafterend%
}

\newcommand{\MHM}{\operatorname{MHM}}

\newcommand{\Dmod}{\mathcal{D}}
\newcommand{\Mmod}{\mathcal{M}}
\newcommand{\Nmod}{\mathcal{N}}

\newcommand{\derR}{\mathbb{R}}
\newcommand{\derL}{\mathbb{L}}

\newcommand{\decal}[1]{\lbrack #1 \rbrack}

\newcommand{\ltriangle}[4][]%
{\begin{diagram}[#1]%
	{#2} &\rTo& {#3} &\rTo& {#4} &\rTo& {#2 \decal{1}}%
\end{diagram}}

\newcommand{\shH}{\mathcal{H}}

\newcommand{\abs}[1]{\lvert #1 \rvert}

\newcommand{\tensor}{\otimes}

\newcommand{\shHom}{\mathcal{H}\hspace{-1pt}\mathit{om}}

\newcommand{\ZZ}{\mathbb{Z}}
\newcommand{\QQ}{\mathbb{Q}}

\newcommand{\CC}{\mathbb{C}}
\newcommand{\HH}{\mathbb{H}}

\newcommand{\menge}[2]{\bigl\{ \thinspace #1 \thinspace\thinspace \big\vert%
\thinspace\thinspace #2 \thinspace \bigr\}}
\newcommand{\Menge}[2]{\Bigl\{ \thinspace #1 \thinspace\thinspace \Big\vert%
\thinspace\thinspace #2 \thinspace \Bigr\}}

\DeclareMathOperator{\im}{im}

\DeclareMathOperator{\Spec}{Spec}

\DeclareMathOperator{\id}{id}

\DeclareMathOperator{\Supp}{Supp}
\DeclareMathOperator{\codim}{codim}

\DeclareMathOperator{\rat}{rat}

\DeclareMathOperator{\Sym}{Sym}
\DeclareMathOperator{\gr}{gr}
\DeclareMathOperator{\DR}{DR}

\newcommand{\define}[1]{\emph{#1}}

\newcommand{\lie}[2]{\lbrack #1, #2 \rbrack}

\newcommand{\shf}[1]{\mathscr{#1}}
\newcommand{\OX}{\shf{O}_X}

\newcommand{\argbl}{-}

\def\overbar#1#2#3{{%
	\setbox0=\hbox{\displaystyle{#1}}%
	\dimen0=\wd0
	\advance\dimen0 by -#2 
	\vbox {\nointerlineskip \moveright #3 \vbox{\hrule height 0.3pt width \dimen0}%
		\nointerlineskip \vskip 1.5pt \box0}%
}}

\newcommand{\into}{\hookrightarrow}

\DeclareMathOperator{\GL}{GL}

\newcommand{\pil}{\pi_{\ast}}

\newcommand{\fu}{f^{\ast}}
\newcommand{\fl}{f_{\ast}}

\newcommand{\iu}{i^{\ast}}

\newcommand{\il}{i_{\ast}}

\newcommand{\pl}{p_{\ast}}

\newcommand{\shF}{\shf{F}}

\newcommand{\shE}{\shf{E}}

\newcommand{\shO}{\shf{O}}

\usepackage{chngcntr}

\makeatletter
\let\@@seccntformat\@seccntformat
\renewcommand*{\@seccntformat}[1]{%
  \expandafter\ifx\csname @seccntformat@#1\endcsname\relax
    \expandafter\@@seccntformat
  \else
    \expandafter
      \csname @seccntformat@#1\expandafter\endcsname
  \fi
    {#1}%
}
\newcommand*{\@seccntformat@subsection}[1]{%
  \textbf{\csname the#1\endcsname.}
}
\makeatother

\makeatletter
\let\@paragraph\paragraph
\renewcommand*{\paragraph}[1]{%
	\vspace{0.3\baselineskip}%
	\@paragraph{\textit{#1}}%
}
\makeatother

\counterwithin{equation}{subsection}
\counterwithout{subsection}{section}
\counterwithin{figure}{subsection}

\newcommand{\subsecref}[1]{\S\ref{#1}}

\newtheorem{theorem}[equation]{Theorem}
\newtheorem*{theorem*}{Theorem}
\newtheorem{lemma}[equation]{Lemma}
\newtheorem*{lemma*}{Lemma}
\newtheorem{corollary}[equation]{Corollary}
\newtheorem{proposition}[equation]{Proposition}
\newtheorem*{proposition*}{Proposition}

\theoremstyle{definition}
\newtheorem{definition}[equation]{Definition}
\newtheorem*{definition*}{Definition}
\theoremstyle{remark}

\newtheorem*{question}{Question}

\newtheorem{example}[equation]{Example}
\newtheorem*{example*}{Example}
\newtheorem*{note}{Note}

\theoremstyle{plain}

\makeatletter
\let\old@caption\caption
\renewcommand*{\caption}[1]{%
	\setcounter{figure}{\value{equation}}%
	\stepcounter{equation}%
	\old@caption{#1}\relax%
}
\makeatother

\newcounter{thmA}
\newtheorem{theorem-intro}[thmA]{Theorem}

\newcommand{\OA}{\shf{O}_A}
\newcommand{\OAh}{\shf{O}_{\Ah}}
\newcommand{\OAsh}{\shf{O}_{\Ash}}
\newcommand{\OBsh}{\shf{O}_{\Bsh}}

\newcommand{\Ash}{A^{\natural}}
\newcommand{\Bsh}{B^{\natural}}
\newcommand{\fsh}{f^{\natural}}
\newcommand{\fp}{f_{+}}
\newcommand{\fsi}{f^{+}}

\newcommand{\Dbcoh}{\mathrm{D}_{\mathit{coh}}^{\mathit{b}}}
\newcommand{\Db}{\mathrm{D}^{\mathit{b}}}
\newcommand{\Dtcoh}[1]{\mathrm{D}_{\mathit{coh}}^{#1}}

\newcommand{\pDtcoh}[2]{ {^{#1}} \Dtcoh{#2}}
\newcommand{\pCoh}[2]{ {^{#1}} \mathrm{Coh}{#2}}
\newcommand{\piDtc}[1]{ {^\pi} \mathrm{D}_c^{#1}}

\newcommand{\mCoh}{ {^m} \mathrm{Coh}}
\newcommand{\mDtcoh}[1]{ {^m} \Dtcoh{#1}}

\newcommand{\Dbh}{\Db_{\mathit{h}}}
\newcommand{\Dbc}{\Db_{\mathit{c}}}
\newcommand{\Dbrh}{\Db_{\mathit{rh}}}

\newcommand{\Dth}[1]{\mathrm{D}_h^{#1}}

\renewcommand{\derR}{\mathbf{R}}
\renewcommand{\derL}{\mathbf{L}}
\newcommand{\Ltensor}{\stackrel{\derL}{\tensor}}
\renewcommand{\argbl}{-}

\renewcommand{\shHom}{\mathcal{H}\mathit{om}}
\newcommand{\Ah}{\widehat{A}}

\DeclareMathOperator{\Pic}{Pic}

\newcommand{\OmA}[1]{\Omega_A^{#1}}
\newcommand{\DA}{\mathbf{D}_A}

\newcommand{\shC}{\mathscr{C}}

\newcommand{\CH}{\operatorname{\mathit{CH}}}

\newcommand{\shA}{\mathcal{A}}

\newcommand{\mm}{\mathfrak{m}}

\DeclareMathOperator{\Char}{Char}
\newcommand{\CCrho}{\CC_{\rho}}
\newcommand{\Psh}{P^{\natural}}
\newcommand{\nablash}{\nabla^{\natural}}
\DeclareMathOperator{\FM}{FM}
\DeclareMathOperator{\FMt}{\widetilde{FM}}
\newcommand{\Cst}{\CC^{\ast}}
\newcommand{\shL}{\mathcal{L}}
\newcommand{\mmrho}{\mm_{\rho}}
\renewcommand{\shH}{\mathcal{H}}
\renewcommand{\HH}{\mathbf{H}}
\newcommand{\opp}{\mathit{opp}}

\newcommand{\shI}{\mathcal{I}}

\newcommand{\shT}{\mathscr{T}}

\renewcommand{\Dmod}{\mathscr{D}}
\newcommand{\Rmod}{\mathscr{R}}

\newcommand{\Pg}{\widetilde{P}}
\newcommand{\nablag}{\widetilde{\nabla}}

\newcommand{\QQb}{\bar{\QQ}}
\DeclareMathOperator{\Aut}{Aut}
\newcommand{\Asig}{A^{\sigma}}
\newcommand{\Asigsh}{(A^{\sigma})^{\natural}}
\newcommand{\Lsig}{L^{\sigma}}
\newcommand{\nablasig}{\nabla^{\sigma}}
\newcommand{\csig}{c_{\sigma}}

\newcommand{\Msig}{M_{\sigma}}
\newcommand{\Mmodsig}{\Mmod_{\sigma}}

\newenvironment{renumerate}{%
	\begin{enumerate}[label=(\roman{*}), ref=(\roman{*})]
}{%
	\end{enumerate}%
}
\newenvironment{aenumerate}{%
	\begin{enumerate}[label=(\alph{*}), ref=(\alph{*})]
}{%
	\end{enumerate}%
}

\newcommand{\Mmodt}{\widetilde{\Mmod}}
\newcommand{\ReesM}{\widetilde{R_F \Mmod}}

\begin{document}

%========================================================
\title[Holonomic complexes on abelian varieties]%
	{Holonomic complexes on abelian varieties, Part I}
\author{Christian Schnell}
\address{%
	Institute for the Physics and Mathematics of the Universe \\
	The University of Tokyo \\
	5-1-5 Kashiwanoha, Kashiwa-shi \\
	Chiba 277-8583, Japan
}

\email{christian.schnell@ipmu.jp}

%\subjclass[2000]{}

\keywords{Abelian variety, D-module, holonomic complex, constructible complex,
Fourier-Mukai transform, cohomology support loci, perverse coherent sheaf}

\begin{abstract}
We study the Fourier-Mukai transform for holonomic $\Dmod$-modules on a
complex abelian variety. Among other things, we show that the cohomology support loci
of a holonomic complex are finite unions of translates of triple tori, the
translates being by torsion points for objects of geometric origin; and that the
standard t-structure on the derived category of holonomic complexes corresponds,
under the Fourier-Mukai transform, to a certain perverse coherent t-structure in the
sense of Kashiwara and Arinkin-Bezrukavnikov.
\end{abstract}
\maketitle
%========================================================

\section{Overview of the paper}

\subsection{Introduction}

This is the first in a series of papers about holonomic $\Dmod$-modules on complex
abelian varieties; the ultimate goal of the series is, in a nutshell, the
description of holonomic $\Dmod$-modules in terms of the Fourier-Mukai transform.

Let $A$ be a complex abelian variety, and let $\Dmod_A$ be the sheaf of linear
differential operators. The most basic examples of left $\Dmod_A$-modules are line
bundles $L$ with integrable connection $\nabla \colon L \to \Omega_A^1 \tensor L$.
Because $A$ is an abelian variety, the moduli space $\Ash$ of all such
is a quasi-projective algebraic variety of dimension $2 \dim A$. The main
idea in the study of $\Dmod_A$-modules is to exploit the fact that $\Ash$
is so big. 

One approach is to consider, for a left $\Dmod_A$-module $\Mmod$, the cohomology
groups (in the sense of $\Dmod$-modules) of the various twists $\Mmod \tensor_{\OA}
(L, \nabla)$. This information may be presented using the \define{cohomology
support loci} of $\Mmod$, which are the sets
\begin{equation} \label{eq:CSL-h}
	S_m^k(\Mmod) = \Menge{(L, \nabla) \in \Ash}{\dim \HH^k \Bigl( A, 
		\DR_A \bigl( \Mmod \tensor_{\OA} (L, \nabla) \bigr) \Bigr) \geq m}.
\end{equation}

Another way to present the information about cohomology of twists of $\Mmod$ is
through the \define{Fourier-Mukai transform} for algebraic $\Dmod_A$-modules, which was
introduced and studied by Laumon \cite{Laumon} and Rothstein \cite{Rothstein}. It is
an exact functor
\begin{equation} \label{eq:FM-intro}
	\FM_A \colon \Dbcoh(\Dmod_A) \to \Dbcoh(\OAsh),
\end{equation}
defined as the integral transform with kernel $(\Psh, \nablash)$, the tautological line
bundle with relative integrable connection on $A \times \Ash$.
As shown by Laumon and Rothstein, $\FM_A$ is an equivalence between the bounded
derived category
of coherent algebraic $\Dmod_A$-modules, and that of algebraic coherent sheaves on
$\Ash$. In essence, this means that an algebraic $\Dmod$-module on an
abelian variety can be recovered from the cohomology of its twists by line bundles
with integrable connection. 

Now the most interesting $\Dmod$-modules are clearly
the holonomic ones; recall that a $\Dmod$-module is \define{holonomic} if its
characteristic variety is a Lagrangian subset of the cotangent bundle. One reason is
that, via the Riemann-Hilbert correspondence, the category of regular holonomic
$\Dmod$-modules is equivalent to the category of perverse sheaves. The motivation for
this study is the following natural question:

\begin{question}
Let $\Dbh(\Dmod_A)$ denote the full subcategory of $\Dbcoh(\Dmod_A)$, consisting of
complexes with holonomic cohomology sheaves. What is the image of $\Dbh(\Dmod_A)$
under the Fourier-Mukai transform? In particular, what does the complex
of sheaves $\FM_A(\Mmod)$ look like when $\Mmod$ is a holonomic $\Dmod_A$-module?
\end{question}

In this paper, we prove several results about
cohomology support loci and Fourier-Mukai transforms of holonomic complexes of
$\Dmod_A$-modules. Among other things, we establish a fundamental structure theorem
for cohomology support loci, and show that the Fourier-Mukai transform of a holonomic
$\Dmod_A$-module satisfies certain codimension inequalities that very much resemble
the famous generic vanishing theorem of Green and Lazarsfeld \cite{GL1}. 

In fact, the similarities with generic vanishing theory are no accident. As explained
in \cite{PS}, generic vanishing theory is really a collection of statements about
certain holonomic $\Dmod$-modules on abelian varieties (namely those that are
obtained as direct images of structure sheaves of irregular smooth projective
varieties) and natural filtrations on them; quite surprisingly, it turns out that all
statements that do not involve the filtration are actually true for arbitrary
holonomic $\Dmod$-modules. This study should therefore be viewed as a continuation
and generalization of the work of Green-Lazarsfeld \cites{GL1,GL2}, Arapura \cite{Arapura},
and Simpson \cite{Simpson}.

\begin{note}
Before we proceed, a few remarks about related works may be helpful.
\begin{enumerate}
\item Some of the results in this paper have been announced in \cite{PS}*{\S26}.
\item The main object of \cite{PS} are filtered $\Dmod_A$-modules underlying
mixed Hodge modules on $A$, and in particular, the individual coherent sheaves in the
filtration. We do not prove any result of that type in this paper.
\item For regular holonomic $\Dmod$-modules (or equivalently, for perverse sheaves),
similar but generally weaker statements were also obtained by Kr\"amer and Weissauer
\cites{KW,Weissauer}. 
\item All our results here also hold, suitably interpreted, for
holonomic $\Dmod$-modules on compact complex tori. This will be explained elsewhere.
\end{enumerate}
\end{note}

In Part~II of the series, we will show (using the work of Mochizuki and Sabbah) that
the Fourier-Mukai transform of a semisimple holonomic $\Dmod_A$-module is locally
analytically represented by a complex with linear differentials. In Part~III, we
shall (hopefully) answer the above question, by identifying the image of the category of
holonomic $\Dmod_A$-modules with a suitably defined category of ``hyperk\"ahler perverse
sheaves'' on the noncompact hyperk\"ahler manifold $\Ash$.

\subsection{Results about constructible complexes}
\label{subsec:results-c}

Although the focus of this work is on holonomic $\Dmod$-modules on abelian varieties,
we shall begin by describing the main results in the more familiar setting of
constructible complexes. Proofs for all the theorems in this section may be found in
\subsecref{subsec:proofs}.

First a few words about the terminology. By a \define{constructible complex} on the
abelian variety $A$, we mean a complex $E$ of
sheaves of $\CC$-vector spaces, whose cohomology sheaves $\shH^i E$ are
constructible with respect to an algebraic (or equivalently, complex analytic)
stratification of $A$, and vanish for $i$ outside some bounded
interval. We denote by $\Dbc(\CC_A)$ the bounded derived category of constructible
complexes. A basic fact \cite{HTT}*{Section~4.5} is that the hypercohomology groups
$\HH^i(A, E)$ are finite-dimensional complex vector spaces for any $E \in
\Dbc(\CC_A)$.

Now let $\Char(A)$ be the space of rank one characters of the fundamental group; it is
also the moduli space for local systems of rank one, and for any character $\rho \colon
\pi_1(A) \to \Cst$, we denote the corresponding local system on $A$ by the symbol
$\CCrho$. It is easy to see that $E \tensor_{\CC} \CCrho$ is again constructible for
any $E \in \Dbc(\CC_A)$; we may therefore define the \define{cohomology support loci}
of $E \in \Dbc(\CC_A)$ to be the subsets
\begin{equation} \label{eq:CSL-c}
	S_m^k(E) = \Menge{\rho \in \Char(A)}{\dim \HH^k \bigl( A, E \tensor_{\CC}
		\CCrho \bigr) \geq m},
\end{equation}
for any pair of integers $k,m \in \ZZ$. Since the space of rank one characters is
very large -- its dimension is equal to $2 \dim A$ -- these loci
should contain a lot of information about the original constructible complex $E$, and
indeed they do.

Our first result is the following structure theorem for cohomology support loci.

\begin{theorem} \label{thm:c-linear}
Let $E \in \Dbc(\CC_A)$ be a constructible complex.
\begin{aenumerate}
\item Each $S_m^k(E)$ is a finite union of linear subvarieties of $\Char(A)$.
\label{en:c-linear-a}
\item If $E$ is a semisimple perverse sheaf of geometric origin \cite{BBD}*{6.2.4},
then these linear subvarieties are arithmetic.
\label{en:c-linear-b}
\end{aenumerate}
\end{theorem}

Here we are using the expression \define{(arithmetic) linear subvarieties} for what Simpson
was originally calling \define{(torsion) translates of triple tori} in
\cite{Simpson}*{p.~365}; the precise definition is the following.

\begin{definition}
A \define{linear subvariety} of $\Char(A)$ is any subset of the form 
\begin{equation} \label{eq:lin-CharA}
	\rho \cdot \im \bigl( \Char(f) \colon \Char(B) \to \Char(A) \bigr),
\end{equation}
for a surjective morphism of abelian varieties $f \colon A \to B$ with connected
fibers, and a character $\rho \in \Char(A)$. We say that a linear subvariety is
\define{arithmetic} if $\rho$ can be taken to be torsion point of $\Char(A)$.
\end{definition}

\begin{note}
The reason for the term \define{arithmetic} is as follows. Let $\Ash$ be the moduli
space of line bundles with integrable connection on $A$; it is also a complex
algebraic variety, biholomorphic to $\Char(A)$, but with a different algebraic
structure. When $A$ is defined over a number field, the torsion points are precisely
those points on the algebraic varieties $\Char(A)$ and $\Ash$ that are simultaneously
defined over a number field in both \cite{Simpson}*{Proposition~3.4}.
\end{note}

Our next result has to do with the codimension of the cohomology support loci $S^k(E)
= S_1^k(E)$. Recall that the category $\Dbc(\CC_A)$ has a nonstandard t-structure
\[
	\Bigl( \piDtc{\leq 0}(\CC_A), \piDtc{\geq 0}(\CC_A) \Bigr),
\]
called the \define{perverse t-structure}, whose heart is the abelian category of
perverse sheaves \cite{BBD}. We show that the position of a constructible
complex with respect to this t-structure can be read off from its cohomology
support loci.

\begin{theorem} \label{thm:c-t}
Let $E \in \Dbc(\CC_A)$ be a constructible complex.
\begin{aenumerate}
\item \label{en:c-t-1}
One has $E \in \piDtc{\leq 0}(\CC_A)$ iff $\codim S^k(E) \geq 2 k$ for every $k
\in \ZZ$.
\item \label{en:c-t-2}
Similarly, $E \in \piDtc{\geq 0}(\CC_A)$ iff $\codim S^k(E) \geq -2k$ for every
$k \in \ZZ$.
\item \label{en:c-t-3}
In particular, $E$ is a perverse sheaf iff $\codim S^k(E) \geq \abs{2k}$ for
every $k \in \ZZ$.
\end{aenumerate}
\end{theorem}

A consequence is the following ``generic vanishing theorem'' for perverse sheaves; a
similar (but less precise) statement has also been proved recently by Kr\"amer and
Weissauer \cite{KW}*{Theorem~2}.

\begin{corollary}
Let $E \in \Dbc(\CC_A)$ be a perverse sheaf on a complex abelian variety. Then 
the cohomology support loci $S_m^k(E)$ are finite unions of linear subvarieties of
$\Char(A)$ of codimension at least $\abs{2k}$. In particular, one has
\[
	\HH^k \bigl( A, E \tensor_{\CC} \CCrho \bigr) = 0 
\]
for general $\rho \in \Char(A)$ and $k \neq 0$.
\end{corollary}

The generic vanishing theorem implies that the \define{Euler characteristic}
\[
	\chi(A, E) = \sum_{k \in \ZZ} (-1)^k \dim \HH^k \bigl( A, E \bigr)
\]
of a perverse sheaf on an abelian variety is always nonnegative, a result originally
due to Franecki and Kapranov \cite{FK}*{Corollary~1.4}. Indeed, from the deformation
invariance of the Euler characteristic, we get
\[
	\chi(A, E) = \chi \bigl( A, E \tensor_{\CC} \CCrho \bigr) = 
		\dim \HH^0 \bigl( A, E \tensor_{\CC} \CCrho \bigr) \geq 0
\]
for a generic character $\rho \in \Char(A)$. For \emph{simple} perverse
sheaves with $\chi(A, E) = 0$, we have the following structure theorem, which has
also been proved by Weissauer \cite{Weissauer}*{Theorem~2}.

\begin{theorem} \label{thm:c-simple}
Let $E \in \Dbc(\CC_A)$ be a simple perverse sheaf. If $\chi(A, E) = 0$, then there
exists an abelian variety $B$, a surjective morphism $f \colon A \to B$ with connected
fibers, and a simple perverse sheaf $E' \in \Dbc(\CC_B)$ with $\chi(B, E') > 0$, such that
\[
	E = \fu E' \tensor_{\CC} \CCrho
\]
for some character $\rho \in \Char(A)$.
\end{theorem}

\subsection{Results about holonomic complexes}

All the theorems in the previous section are actually consequences of similar results
about holonomic complexes of $\Dmod$-modules on abelian varieties. In fact, the
situation for $\Dmod_A$-modules is considerably better, because we have the
Fourier-Mukai transform \eqref{eq:FM-intro} at our disposal.

Again, we begin by saying a few words about terminology.
Recall that $\Dmod_A$ is the sheaf of linear differential operators
of finite order; since the tangent bundle of $A$ is trivial, $\Dmod_A$ is generated,
as an $\OA$-algebra, by any basis $\partial_1, \dotsc, \partial_g \in H^0(A,
\shT_A)$, subject to the relations $\lie{\partial_i}{\partial_j} = 0$ and $\lie{\partial_i}{f}
= \partial_i f$. By an (algebraic) \define{$\Dmod_A$-module}, we mean a sheaf of left
$\Dmod_A$-modules that is quasi-coherent as a sheaf of $\OA$-modules; a
$\Dmod_A$-module is holonomic if its characteristic variety, as a subset of the
cotangent bundle $T^{\ast} A$, has dimension equal to $\dim A$. Finally, a
\define{holonomic complex} is a complex of $\Dmod_A$-modules $\Mmod$, whose
cohomology sheaves $\shH^i \Mmod$ are holonomic, and vanish for $i$ outside some
bounded interval. We denote by $\Dbcoh(\Dmod_A)$ the bounded derived
category of coherent $\Dmod_A$-modules, and by $\Dbh(\Dmod_A)$ the full subcategory
of holonomic complexes.

Let $\Ash$ be the moduli space of line bundles with integrable connection on $A$.
For any $\Dmod_A$-module $\Mmod$, and any $(L, \nabla) \in \Ash$, the tensor product
$\Mmod \tensor_{\OA} L$ again has the structure of a $\Dmod_A$-module; for the sake
of clarity, we will denote it by the symbol $\Mmod \tensor_{\OA} (L, \nabla)$. We
then define the cohomology support loci of a complex of $\Dmod_A$-modules 
$\Mmod \in \Dbcoh(\Dmod_A)$ by the same formula as in \eqref{eq:CSL-h}, where 
where $\DR_A$ denotes the de Rham functor.

One of the main results of this paper is the following structure theorem for
cohomology support loci of holonomic complexes.

\begin{definition}
A \define{linear subvariety} of $\Ash$ is any subset of the form 
\begin{equation} \label{eq:lin-Ash}
	(L, \nabla) \tensor \im \bigl( \fsh \colon \Bsh \to \Ash \bigr),
\end{equation}
for a surjective morphism of abelian varieties $f \colon A \to B$ with connected
fibers, and a line bundle with integrable connection $(L, \nabla) \in \Ash$. We say
that a linear subvariety is \define{arithmetic} if $(L, \nabla)$ can be taken to be torsion
point of $\Ash$.
\end{definition}

\begin{theorem} \label{thm:h-linear}
Let $\Mmod \in \Dbh(\Dmod_A)$ be a holonomic complex.
\begin{aenumerate}
\item Each $S_m^k(\Mmod)$ is a finite union of linear subvarieties of $\Ash$.
\item If $\Mmod$ is a semisimple regular holonomic $\Dmod_A$-module of geometric
origin, in the sense of \cite{BBD}*{6.2.4}, then these linear subvarieties are
arithmetic.  
\end{aenumerate}
\end{theorem}

Theorem~\ref{thm:h-linear} and Theorem~\ref{thm:c-linear} are proved together; in fact,
the main idea is to exploit the close relationship between cohomology support loci for
constructible and holonomic complexes. Recall that we have a biholomorphic mapping
\begin{equation} \label{eq:Phi}
	\Phi \colon \Ash \to \Char(A), \quad (L, \nabla) \mapsto \ker \nabla,
\end{equation}
by the well-known correspondence between local systems and vector bundles with
integrable connection. Now if $\Mmod$ is a holonomic $\Dmod_A$-module, then according
to a theorem of Kashiwara \cite{HTT}*{Theorem~4.6.6}, its de Rham complex
\[
	\DR_A(\Mmod) = \Bigl\lbrack \Mmod \to \Omega_A^1 \tensor \Mmod \to 
		\dotsb \to \Omega_A^{\dim A} \tensor \Mmod \Bigr\rbrack,
\]
placed in degrees $-\dim A, \dotsc, 0$, is a perverse sheaf on $A$. More generally,
$\DR_A(\Mmod)$ is a constructible complex for any $\Mmod \in \Dbh(\Dmod_A)$
\cite{HTT}*{Theorem~4.6.3}. It is very easy to show -- see
Lemma~\ref{lem:relationship} below -- that the cohomology support loci for $\Mmod$
and $\DR_A(\Mmod)$ are connected by the formula 
\begin{equation} \label{eq:relation}
	\Phi \bigl( S_m^k(\Mmod) \bigr) = S_m^k \bigl( \DR_A(\Mmod) \bigr).
\end{equation}
A much deeper relationship comes from the Riemann-Hilbert correspondence of Kashiwara
and Mebkhout \cite{HTT}*{Theorem~7.2.1}, which asserts that the functor
\[
	\DR_A \colon \Dbrh(\Dmod_A) \to \Dbc(\CC_A)
\]
from regular holonomic complexes to constructible complexes is an equivalence
of categories. Together with \eqref{eq:relation}, this means that cohomology support
loci for holonomic and constructible complexes completely determine each other.

Let us now briefly sketch the proof of Theorem~\ref{thm:h-linear}. Our starting point
is the observation that both complex manifolds $\Char(A)$ and $\Ash$ naturally have
the structure of complex algebraic varieties; while isomorphic as complex
manifolds, they are not isomorphic as algebraic varieties. The special property of
\emph{linear subvarieties} is that they are algebraic in both models. Indeed, for any
surjective morphism $f \colon A \to B$ of abelian varieties, \eqref{eq:lin-Ash} is an
algebraic subvariety of $\Ash$, and \eqref{eq:lin-CharA} is an algebraic subvariety
of $\Char(A)$; moreover, since $\Phi$ is an isomorphism of complex Lie groups,
\[
	\Phi \Bigl( (L, \nabla) \tensor 
		\im \bigl( \fsh \colon \Bsh \to \Ash \bigr) \Bigr) = 
		 \Phi(L, \nabla) \cdot \im \bigl( \Char(f) \colon \Char(B) \to \Char(A) \bigr).
\]
The following difficult theorem by Simpson \cite{Simpson}*{Theorem~3.1} shows that
finite unions of linear subvarieties are the only closed subsets with this property.

\begin{theorem}[Simpson] \label{thm:simpson}
Let $Z$ be a closed algebraic subset of $\Ash$. If $\Phi(Z)$ is again a closed
algebraic subset of $\Char(A)$, then $Z$ is a finite union of linear subvarieties of
$\Ash$, and $\Phi(Z)$ is a finite union of linear subvarieties of $\Char(A)$.
\end{theorem}

Thus it suffices to show that cohomology support loci are algebraic subsets of
$\Char(A)$ and $\Ash$, which we do in Theorem~\ref{thm:alg-c} and
Proposition~\ref{prop:alg-h} below. The argument in \subsecref{subsec:constructible}
is based on a sort of ``Fourier-Mukai transform'' for constructible complexes, which
may be of independent interest.

Theorem~\ref{thm:h-linear} has the following consequence for the support of the
Fourier-Mukai transform of a holonomic complex.

\begin{corollary} \label{cor:FM-linear}
Let $\Mmod \in \Dbh(\Dmod_A)$ be a holonomic complex on an abelian variety. Then the
support of the complex $\FM_A(\Mmod)$ is a finite union of linear subvarieties of
$\Ash$. These linear subvarieties are arithmetic whenever $\Mmod$ is a semisimple
regular holonomic $\Dmod_A$-module of geometric origin.
\end{corollary}

Our second main result is that that the position of a holonomic complex $\Mmod$ with
respect to the standard t-structure on $\Dbh(\Dmod_A)$ can be read off from the
codimension of its cohomology support loci $S^k(\Mmod) = S_1^k(\Mmod)$.

\begin{theorem}  \label{thm:h-t}
Let $\Mmod \in \Dbh(\Dmod_A)$ be a holonomic complex.
\begin{aenumerate}
\item 
One has $\Mmod \in \Dth{\leq 0}(\Dmod_A)$ iff $\codim S^k(\Mmod) \geq 2 k$ for every $k
\in \ZZ$.
\item 
Similarly, $\Mmod \in \Dth{\geq 0}(\Dmod_A)$ iff $\codim S^k(\Mmod) \geq -2k$ for every
$k \in \ZZ$.
\item
In particular, $\Mmod$ is a single holonomic $\Dmod_A$-module iff $\codim S^k(\Mmod)
\geq \abs{2k}$ for every $k \in \ZZ$.
\end{aenumerate}
\end{theorem}

The proof (in \subsecref{subsec:codimension}) uses the structure theorem above,
together with certain properties of the
Fourier-Mukai transform established by Laumon. Theorem~\ref{thm:h-t} can be
reformulated using the theory of \define{perverse coherent sheaves}, developed by 
Kashiwara \cite{Kashiwara} and by Arinkin and Bezrukavnikov \cite{AB}. In fact, there is
a perverse t-structure on $\Dbcoh(\OAsh)$ with the property that
\[
	\mDtcoh{\leq 0}(\OAsh) = \menge{E \in \Dbcoh(\OAsh)}%
		{\text{$\codim \Supp \shH^k E \geq 2k$ for every $k \in \ZZ$}};
\]
it corresponds to the supporting function $m = \left\lfloor \tfrac{1}{2} \codim
\right\rfloor$ on the topological space of the scheme $\Ash$, in Kashiwara's
terminology. Its heart $\mCoh(\OAsh)$ is the abelian category of \define{$m$-perverse
coherent sheaves}.

\begin{theorem} \label{thm:h-t-FM}
Let $\Mmod \in \Dbh(\Dmod_A)$ be a holonomic complex on $A$.
\begin{aenumerate}
\item One has $\Mmod \in \Dth{\leq k}(\Dmod_A)$ iff $\FM_A(\Mmod) \in \mDtcoh{\leq
k}(\OAsh)$.
\item Similarly, $\Mmod \in \Dth{\geq k}(\Dmod_A)$ iff $\FM_A(\Mmod) \in \mDtcoh{\geq
k}(\OAsh)$.
\item In particular, $\Mmod$ is a single holonomic $\Dmod_A$-module iff
its Fourier-Mukai transform $\FM_A(\Mmod)$ is an $m$-perverse coherent sheaf on $\Ash$.
\end{aenumerate}
\end{theorem}

Using basic properties of the $m$-perverse t-structure (see
\subsecref{subsec:perverse}), it follows that the Fourier-Mukai transform of a
holonomic $\Dmod_A$-module is always concentrated in degrees $0, 1, \dotsc, \dim A$.
Moreover, $\shH^i \FM_A(\Mmod)$ is a torsion sheaf for $i > 0$; and for $i = 0$, it is a
torsion sheaf iff it is zero.

For that reason, one would expect $\shH^0 \FM_A(\Mmod)$ to be supported on all of
$\Ash$, but examples show that this fails when $\Mmod$ is pulled back from a
lower-dimensional abelian variety. This suggests our third main result, namely the
following structure theorem for \emph{simple} holonomic $\Dmod_A$-modules.

\begin{theorem} \label{thm:h-simple}
Let $\Mmod$ be a simple holonomic $\Dmod_A$-module. Then there exists an abelian
variety $B$, a surjective morphism $f \colon A \to B$ with connected fibers, and a
simple holonomic $\Dmod_B$-module $\Nmod$ with $\Supp \shH^0 \FM_B(\Nmod) = \Bsh$,
such that 
\[
	\Mmod \simeq \fu \Nmod \tensor_{\OA} (L, \nabla)
\]
for some $(L, \nabla) \in \Ash$.
\end{theorem}

This result again follows from the basic properties of the Fourier-Mukai transform,
together with the following interesting fact about the $m$-perverse t-structure
(see Proposition~\ref{prop:surprise}): If a complex of coherent sheaves $E$ and its dual
complex $\derR \shHom(E, \OAsh)$ both belong to $\mCoh(\OAsh)$, and if $r \in \ZZ$
denotes the least integer with $\shH^r E \neq 0$, then $\codim \Supp \shH^r E = 2r$. 

One application of Theorem~\ref{thm:h-simple} is to describe simple holonomic
$\Dmod_A$-modules with Euler characteristic zero. Recall that the \define{Euler
characteristic} of a coherent algebraic $\Dmod_A$-module $\Mmod$ is the integer
\[
	\chi(A, \Mmod) = \sum_{k \in \ZZ} (-1)^k \dim \HH^k \bigl( A, \DR_A(\Mmod) \bigr).
\]
When $\Mmod$ is holonomic, we have $\chi(A, \Mmod) \geq 0$ as a consequence of
Theorem~\ref{thm:h-t-FM}. In the regular case, the following result has
independently been proved by Weissauer \cite{Weissauer}*{Theorem~2}.
	
\begin{corollary} \label{cor:h-simple}
Let $\Mmod$ be a simple holonomic $\Dmod_A$-module with $\chi(A, \Mmod) = 0$.
Then there is a surjective morphism with connected fibers $f \colon A \to B$ to a
lower-dimensional abelian variety, and a simple holonomic $\Dmod_B$-module $\Nmod$,
such that
\[	
	\Mmod \simeq \fu \Nmod \tensor_{\OA} (L, \nabla)
\]
for a suitable point $(L, \nabla) \in \Ash$. Moreover, we may assume that $\chi(B,
\Nmod) > 0$.
\end{corollary}

\subsection{Acknowledgements}

This work was supported by the World Premier International Research Center Initiative
(WPI Initiative), MEXT, Japan, and by NSF grant DMS-1100606.  I thank Mihnea Popa and
Pierre Schapira for their comments about the paper, and Takuro Mochizuki, Kiyoshi
Takeuchi, Giovanni Morando, and Kentaro Hori for useful discussions. I am also very
grateful to my parents-in-law for their hospitality during the preparation of the
manuscript.

\section{The Fourier-Mukai transform}

\subsection{The case of $\Dmod$-modules}

In this section, we recall the definition of the generalized Fourier-Mukai transform,
introduced by Rothstein \cite{Rothstein} and Laumon \cite{Laumon}. It is defined in
the algebraic category, and so we begin by explaining why $\Ash$ is a complex
algebraic variety. As always, let $A$ be a complex abelian variety. The moduli space
$\Ash$ of algebraic line bundles with integrable connection on $A$ naturally has the
structure of a quasi-projective algebraic variety: on the dual abelian variety $\Ah =
\Pic^0(A)$, there is a canonical extension of vector bundles
\begin{equation} \label{eq:extension}
	0 \to \Ah \times H^0(A, \OmA{1}) \to E(A) \to \Ah \times \CC \to 0,
\end{equation}
and $\Ash$ is isomorphic to the preimage of $\Ah \times \{1\}$ inside of $E(A)$.
The projection
\[
	\pi \colon \Ash \to \Ah,  \quad (L, \nabla) \mapsto L,
\] 
is thus a torsor for the trivial bundle $\Ah \times H^0(A, \OmA{1})$; this
corresponds to the fact that $\nabla + \omega$ is again an integrable connection for any
$\omega \in H^0(A, \OmA{1})$. Note that $\Ash$ is a group under tensor product, with
unit element the trivial line bundle $(\OA, d)$.

The generalized Fourier-Mukai transform takes bounded complexes of coherent algebraic
$\Dmod_A$-modules to bounded complexes of algebraic coherent sheaves on $\Ash$;
we briefly describe it following the presentation in \cite{Laumon}*{\S3}. Let $P$
denote the normalized Poincar\'e bundle on the product $A \times \Ah$.
Since $\Ash$ is the moduli space of line bundles with integrable connection on $A$, the
pullback $\Psh = (\id_A \times \pi)^{\ast} P$ of the Poincar\'e bundle to the product
$A \times \Ash$ is endowed with a universal integrable connection 
\[
	\nablash \colon \Psh \to \Omega_{A \times \Ash / \Ash}^1 \tensor \Psh 
\]
relative to $\Ash$.
Given any algebraic left $\Dmod_A$-module $\Mmod$, interpreted as a quasi-coherent
sheaf of $\OA$-modules with integrable connection $\nabla \colon \Mmod \to \Omega_A^1
\tensor \Mmod$, the tensor product $p_1^* \Mmod \tensor_{\OA} (\Psh, \nablash)$
inherits a natural integrable connection relative to $\Ash$. We then define the
\define{Fourier-Mukai transform} of $\Mmod$ by the formula
\begin{equation} \label{eqn:laumon}
	\FM_A(\Mmod) = \derR (p_2)_{\ast} \DR_{A \times \Ash / \Ash}
		 \bigl( \, p_1^* \Mmod \tensor (\Psh, \nablash) \bigr)
\end{equation}
where the relative de Rham complex 
\[
	\Bigl\lbrack
		p_1^{\ast} \Mmod \tensor \Psh \to \Omega^1_{A\times \Ash/ \Ash} 
			\tensor p_1^{\ast} \Mmod \tensor \Psh \to \dotsb \to
 		\Omega^g_{A\times \Ash/ \Ash} \tensor p_1^{\ast} \Mmod \tensor \Psh
	\Bigr\rbrack
\]
is placed in degrees $-g, \ldots,0$ as usual. Since every differential in the complex
is $\OAsh$-linear, it follows that $\FM_A(\Mmod)$ is a complex of
algebraic quasi-coherent sheaves on $\Ash$. Finally, Laumon
\cite{Laumon}*{Th\'eor\`em~3.2.1 and Corollaire~3.2.5} proves that this operation
induces an equivalence
\begin{equation} \label{eq:FM}
	\FM_A \colon \Dbcoh(\Dmod_A) \to \Dbcoh(\OAsh)
\end{equation}
between the bounded derived category of coherent algebraic $\Dmod_A$-modules and 
the bounded derived category of algebraic coherent sheaves on $\Ash$. Rothstein
obtained the same result by a different method in \cite{Rothstein}*{Theorem 6.2}.

\begin{note}
Since $A$ is a complex projective variety, the category of coherent analytic
$\Dmod$-modules on $A$ is equivalent to the category of coherent algebraic
$\Dmod$-modules by a version of the GAGA theorem. On the other hand, it is essential
to consider only \emph{algebraic} coherent sheaves on $\Ash$ in \eqref{eq:FM},
because $\Ash$ is not projective.
\end{note}

We list some basic properties of the Fourier-Mukai transform. For $f \colon A \to B$
a surjective morphism of abelian varieties, one has the \define{direct image functor}
\[
	\fp \colon \Dbcoh(\Dmod_A) \to \Dbcoh(\Dmod_B), \qquad
		\fp \Mmod = \derR \fl \DR_{A/B}(\Mmod),
\]
where $\DR_{A/B}(\Mmod)$ denotes the relative de Rham complex 
\[
	\DR_{A/B}(\Mmod) = \Bigl\lbrack
		\Mmod \to \Omega_{A/B}^1 \tensor \Mmod \to \dotsb \to 
			\Omega_{A/B}^r \tensor \Mmod \Bigr\rbrack,
\]
placed in degrees $-r, \dotsc, 0$, and $r = \dim A - \dim B$ is the relative
dimension of $f$. For holonomic complexes, we have an induced functor
\[
	\fp \colon \Dbh(\Dmod_A) \to \Dbh(\Dmod_B)
\]
since direct images under algebraic morphisms preserve holonomicity
\cite{HTT}*{Theorem~3.2.3}.  We also use the \define{shifted inverse image functor}
\[
	\fsi = \derL \fu \decal{\dim A - \dim B} \colon \Dbc(\Dmod_B) \to \Dbc(\Dmod_A),
\]
which again preserves holonomic complexes since $f$ is smooth. Finally, we use the
\define{duality functor}
\[
	\DA \colon \Dbcoh(\Dmod_A) \to \Dbcoh(\Dmod_A)^{\opp}, \quad
		\DA(\Mmod) = \derR \shHom_{\Dmod_A} \bigl( \Mmod, 
		(\Omega_A^{g})^{-1} \tensor \Dmod_A \bigr) \decal{g}.
\]
Note that a $\Dmod_A$-module $\Mmod$ is holonomic exactly when $\DA(\Mmod)$ is
again a $\Dmod_A$-module, viewed as a complex concentrated in degree zero.

\begin{theorem}[Laumon] \label{thm:Laumon}
The Fourier-Mukai transform for $\Dmod$-modules on abelian varieties has the
following properties.
\begin{aenumerate}
\item \label{en:Laumon1}
For $(L, \nabla) \in \Ash$, denote by $t_{(L, \nabla)} \colon \Ash \to \Ash$ the
translation morphism. Then one has a canonical and functorial isomorphism
\[
	\FM_A \bigl( \argbl \tensor_{\OA} (L, \nabla) \bigr) =
		\derL (t_{(L, \nabla)})^{\ast} \circ \FM_A.
\]
\item \label{en:Laumon4}
One has a canonical and functorial isomorphism
\[
	\FM_A \circ \DA = \langle -1_{\Ash} \rangle^{\ast} \, 
		\derR \shHom \bigl( \FM_A(\argbl), \OAsh \bigr).
\]
\item \label{en:Laumon2}
For a surjective morphism $f \colon A \to B$, denote by $\fsh \colon \Bsh \to \Ash$
the induced morphism. Then one has a canonical and functorial isomorphism
\[
	\derL (\fsh)^{\ast} \circ \FM_A = \FM_B \circ \fp.
\]
\item \label{en:Laumon3}
In the same situation, one has a canonical and functorial isomorphism
\[
	\derR \fsh_{\ast} \circ \FM_B = \FM_A \circ \fsi.
\]
\end{aenumerate}
\end{theorem}

\begin{proof}
\ref{en:Laumon1} follows immediately from the properties of the Poincar\'e bundle on
$A \times \Ash$. \ref{en:Laumon2} and \ref{en:Laumon3} are proved in
\cite{Laumon}*{Proposition~3.3.2}; note that ``$g-1-g_2$'' should read ``$g_1 -
g_2$.'' Lastly, \ref{en:Laumon4} is contained in \cite{Laumon}*{Proposition~3.3.4}.
\end{proof}

\subsection{The Rees algebra}

Let $A$ be a complex abelian variety of dimension $g$. The sheaf $\Dmod_A$ of linear
differential operators is naturally filtered by the order of
differential operators, and we consider the associated Rees algebra
\[
	\Rmod_A = R_F \Dmod_A = \bigoplus_{k=0}^{\infty} F_k \Dmod_A \cdot z^k
		\subseteq \Dmod_A \tensor_{\OA} \OA \lbrack z \rbrack.
\]
More concretely, let $\partial_1, \dotsc, \partial_g \in H^0(A, \shT_A)$ be a basis for
the space of tangent vector fields on $A$; then as a sheaf of algebras, $\Rmod_A$ is
generated over $\OA \lbrack z \rbrack$ by the operators $\delta_1, \dotsc, \delta_g$, where
$\delta_i = z \partial_i$, subject to the relations
\[
	\lbrack \delta_i, \delta_j \rbrack = 0 \qquad \text{and} \qquad
		\lbrack \delta_i, f \rbrack = z \cdot \partial_i f.
\]
It is easy to see that we have $\Rmod_A / (z-1) \Rmod_A \simeq \Dmod_A$, and $\Rmod_A / 
z \Rmod_A \simeq \Sym \shT_A$.

\begin{definition}
An algebraic \define{$\Rmod_A$-module} is a sheaf of left $\Rmod_A$-modules that is
quasi-coherent over $\OA$. An $\Rmod_A$-module is called \define{strict} if it has no
$z$-torsion.
\end{definition}

\begin{example}
Let $(\Mmod, F)$ be a filtered $\Dmod_A$-module; then the Rees module
\[
	R_F \Mmod = \bigoplus_{k \in \ZZ} F_k \Mmod \cdot z^k
\]
is a strict $\Rmod_A$-module; it is coherent over $\Rmod_A$ iff the filtration $F =
F_{\bullet} \Mmod$ is good.
\end{example}

An equivalent point of view is the following. On the product 
\[
	A \times \CC = A \times_{\Spec \CC} \Spec \CC \lbrack z \rbrack,
\]
consider the subsheaf of $\Dmod_{A \times \CC / \CC}$ generated by $z \shT_{A \times
\CC / \CC}$. For any quasi-coherent sheaf of $\shO_{A \times
\CC}$-modules with a left action by that sheaf of operators, the pushforward to $A$ is
then naturally an $\Rmod_A$-module. Conversely, any algebraic $\Rmod_A$-module
$\Mmod$ gives rise to a quasi-coherent sheaf $\Mmodt$ on $A \times \CC$ that has
the structure of a left module over the above sheaf of operators. 

Given an $\Rmod_A$-module $\Mmod$, we have a $\CC \lbrack z \rbrack$-linear morphism
of sheaves
\[
	\nabla \colon \Mmodt \to \frac{1}{z} \Omega_{A \times \CC/\CC}^1 
		\tensor_{\shO_{A \times \CC}} \Mmodt, \qquad
		\nabla m = \sum_{i=1}^g \frac{\omega_i}{z} \tensor \delta_i m,
\]
where $\omega_1, \dotsc, \omega_g \in H^0(A, \Omega_A^1)$ is the basis dual to
$\partial_1, \dotsc, \partial_g \in H^0(A, \shT_A)$. The \define{de Rham complex} of
$\Mmod$ is the resulting complex
\begin{equation} \label{eq:DR-R}
	\DR_A(\Mmod) = \Bigl\lbrack 
		\Mmodt \to \frac{1}{z} \Omega_{A \times \CC/\CC}^1 \tensor \Mmodt 
			\to \dotsb \to \frac{1}{z^g} \Omega_{A \times \CC/\CC}^g \tensor \Mmodt,
	\Bigr\rbrack
\end{equation}
placed in degrees $-g, \dotsc, 0$, whose differentials are given by the formula
\[
	\frac{\omega}{z^k} \tensor m \mapsto 
		(-1)^g \frac{d \omega}{z^{k+1}} \tensor zm + (-1)^{g+k} 
			\frac{\omega}{z^k} \wedge \nabla m.
\]

\subsection{The moduli space of generalized connections}

In this section, we introduce the moduli space of generalized connections on $A$; it
will be used in the following section to define a Fourier-Mukai transform for
algebraic $\Rmod_A$-modules. As explained in \cite{Bonsdorff}, the idea of this
construction is originally due to Deligne and Simpson.

\begin{definition}
Let $X$ be a complex manifold, and $\lambda \colon X \to \CC$ a holomorphic function. 
A \define{generalized connection with parameter $\lambda$}, or more briefly a
\define{$\lambda$-connection}, on a locally free sheaf of $\OX$-modules $\shE$ is a
$\CC$-linear morphism of sheaves 
\[
	\nabla \colon \shE \to \Omega_X^1 \tensor_{\OX} \shE
\]
that satisfies the Leibniz rule with parameter $\lambda$, which is to say that
\[
	\nabla(f \cdot s) = f \cdot \nabla s + df \tensor \lambda s
\]
for local sections $f \in \OX$ and $s \in \shE$. A $\lambda$-connection is called
\define{integrable} if its $\OX$-linear curvature operator $\nabla \circ \nabla
\colon \shE \to \Omega_X^2 \tensor_{\OX} \shE$ is equal to zero.
\end{definition}

\begin{example}
An integrable $1$-connection is an integrable connection in the usual sense; an
integrable $0$-connection is the same thing as the structure of a Higgs bundle on
$\shE$.
\end{example}

On an abelian variety $A$, the moduli space $E(A)$ of line bundles with
integrable $\lambda$-connection (for arbitrary $\lambda \in \CC$) may be constructed 
as follows. Observe first that a $\lambda$-connection on a line bundle $L \in
\Pic(A)$ is integrable iff $L \in \Pic^0(A)$. To construct the moduli
space, let $\mm_A \subseteq \OA$ denote the ideal sheaf of the unit element
$0_A \in A$.  Restriction of differential forms induces an isomorphism
\[
	\mm_A / \mm_A^2 = H^0(A, \Omega_A^1) \tensor \OA / \mm_A,
\]
and therefore determines an extension of coherent sheaves
\[
	0 \to H^0(A, \Omega_A^1) \tensor \OA / \mm_A \to \OA / \mm_A^2 
		\to \OA / \mm_A \to 0.
\]
Let $P$ be the normalized Poincar\'e bundle on the product $A \times \Ah$, and denote
by $\derR \Phi_P \colon \Dbcoh(\OA) \to \Dbcoh(\OAh)$ the Fourier-Mukai transform.
Then $\derR \Phi_P \bigl( \OA / \mm_A^2 \bigr)$ is a locally free sheaf $\shE(A)$,
and so we obtain an extension of locally free sheaves 
\begin{equation} \label{eq:ext-sheaf}
	0 \to H^0(A, \Omega_A^1) \tensor \OAh \to \shE(A) \to \OAh \to 0
\end{equation}
on the dual abelian variety $\Ah$. The resulting vector bundle extension
\[
	0 \to \Ah \times H^0(A, \Omega_A^1) \to E(A) \to \Ah \times \CC \to 0
\]
defines an algebraic vector bundle $E(A)$ on $\Ah$; by construction, it comes
with two algebraic morphisms $\pi \colon E(A) \to \Ah$ and $\lambda \colon E(A) \to
\CC$.

Now the claim is that $E(A)$ is the moduli space of line bundles with integrable
$\lambda$-connection, with $\pi \colon E(A) \to \Ah$ the map that takes $(L, \nabla,
\lambda)$ to the underlying line bundle $L$. The following lemma establishes the
existence of a tautological line bundle with generalized connection on $A \times
E(A)$.

\begin{lemma} \label{lem:connection}
Let $\Pg = (\id \times \pi)^{\ast} P$ denote the pullback of the Poincar\'e bundle to
$A \times E(A)$. Then there is a canonical generalized relative connection
\[
	\nablag \colon \Pg \to \Omega_{A \times E(A) / E(A)}^1 \tensor \Pg
\]
that satisfies the Leibniz rule $\nablag(f \cdot s) = f \cdot \nablag s
+ d_{A \times E(A) / E(A)} f \tensor \lambda s$.
\end{lemma}

\begin{proof}
Let $\shI$ denote the ideal sheaf of the diagonal in $A \times A$. Let $Z$ be the
non-reduced subscheme of $A \times A \times E(A)$ defined by the ideal sheaf
$\shI^2 \cdot \shO_{A \times A \times E(A)}$. We have a natural exact sequence
\begin{equation} \label{eq:seq-connection}
	0 \to \Pg \tensor H^0(A, \Omega_A^1) \to (p_{13})_{\ast} \bigl( 
		\shO_Z \tensor p_{23}^{\ast} \Pg \bigr) \to \Pg \to 0,
\end{equation}
and a generalized relative connection is the same thing as a morphism of sheaves
\[
	\Pg \to (p_{13})_{\ast} \bigl( \shO_Z \tensor p_{23}^{\ast} \Pg \bigr)
\]
whose composition with the morphism to $\Pg$ acts as multiplication by $\lambda$. In
fact, there is a canonical choice, which we shall now describe.
Consider the morphism
\[
	f \colon A \times A \to A \times A, \quad f(a, b) = (a, a+b).
\]
Since $f \times \id_{E(A)}$ induces an isomorphism between the first infinitesimal
neighborhood of $A \times \{0_A\} \times E(A)$ and the subscheme $Z$, we have
\begin{align*}
	(f \times \id_{E(A)})^{\ast} \bigl( \shO_Z \tensor p_{23}^{\ast} \Pg \bigr)
		&= p_2^{\ast}(\OA / \mm_A^2) \tensor (m \times \id_{E(A)})^{\ast} \Pg \\
		&= p_2^{\ast}(\OA / \mm_A^2) \tensor 
			p_{13}^{\ast} \Pg \tensor p_{23}^{\ast} \Pg,
\end{align*}
due to the well-known fact that the Poincar\'e bundle satisfies
\[
	(m \times \id_{\Ah})^{\ast} P = p_{13}^{\ast} P \tensor p_{23}^{\ast} P
\]
on $A \times A \times \Ah$. Since $p_{13} \circ (f \times \id_{E(A)}) = p_{13}$, we
conclude that we have
\[
	(p_{13})_{\ast} \bigl( \shO_Z \tensor p_{23}^{\ast} \Pg \bigr)
	= \Pg \tensor p_2^{\ast} \pi^{\ast} \derR \Phi_P \bigl( \OA / \mm_A^2 \bigr)
	= \Pg \tensor p_2^{\ast} \pi^{\ast} \shE(A)
\]
on $A \times E(A)$; more precisely, \eqref{eq:seq-connection} is isomorphic to the
tensor product of $\Pg$ and the pullback of \eqref{eq:ext-sheaf} by $\pi \circ p_2$.

Now the pullback of the exact sequence \eqref{eq:ext-sheaf} to $E(A)$ obviously has a
splitting of the type we are looking for: indeed, the tautological section of
$\pi^{\ast} \shE(A)$ gives a morphism $\shO_{E(A)} \to \pi^{\ast} \shE(A)$ whose
composition with the projection to $\shO_{E(A)}$ is multiplication by $\lambda$. Thus
we obtain a canonical morphism $\Pg \to \Pg \tensor p_2^{\ast} \pi^{\ast} \shE(A)$
and hence, by the above, the desired generalized relative connection.
\end{proof}

\begin{note}
At any closed point $e \in E(A)$, we obtain a $\lambda(e)$-connection
on the line bundle corresponding to $\pi(e) \in \Pic^0(A)$.
Using the properties of the Picard scheme, it is not difficult to show that
$E(A)$ is a fine moduli space, in the obvious sense; but since we do not need this
fact below, we shall omit the proof.
\end{note}

\begin{corollary}
The pullback of the Poincar\'e bundle to $A \times E(A)$ has the structure
of a relative $\Rmod_A$-module, where $z$ act as multiplication by $\lambda$, and
$\delta_i$ as $\nabla_{\partial_i}$.
\end{corollary}

\subsection{The case of $\Rmod$-modules}

We shall now describe a version of the Fourier-Mukai transform that works for
algebraic $\Rmod_A$-modules. Since we only need a very special case in this paper, we
leave a more careful discussion to a future publication. 

Let $\Mmod$ be an $\Rmod_A$-module, and denote by $\Mmodt$ the associated
quasi-coherent sheaf on $A \times \CC$; as explained above, $\Mmodt$ is a left module
over the subalgebra of $\Dmod_{A \times \CC/\CC}$ generated by $z \shT_{A \times
\CC/\CC}$. Thus the tensor product
\[
	(\id \times \lambda)^{\ast} \Mmodt \tensor_{\shO_{A \times E(A)}} \Pg
\]
naturally has the structure of a relative $\Rmod_A$-module on $A \times E(A)$, with
$z(m \tensor s) = (\lambda m) \tensor s = m \tensor \lambda s$ and $\delta_i(m \tensor s) =
(\delta_i m) \tensor s + m \tensor \nablag_{\partial_i} s$. We may therefore consider
the relative de Rham complex
\[
	\DR_{A \times E(A) / E(A)} \Bigl( (\id \times \lambda)^{\ast} \Mmodt
		\tensor_{\shO_{A \times E(A)}} \bigl( \Pg, \nablag \bigr) \Bigr),
\]
which is defined analogously to \eqref{eq:DR-R}.

\begin{definition} \label{def:FM-R}
The \define{Fourier-Mukai transform} of an $\Rmod_A$-module $\Mmod$ is
\[
	\FMt_A(\Mmod) = \derR (p_2)_{\ast} 
		\DR_{A \times E(A) / E(A)} \Bigl( (\id \times \lambda)^{\ast} \Mmodt 
			\tensor_{\shO_{A \times E(A)}} \Pg \Bigr);
\]
it is an object of $\Db \bigl( \shO_{E(A)} \bigr)$, the bounded derived category of
quasi-coherent sheaves on $E(A)$.
\end{definition}

\begin{note}
Using the general formalism in \cite{PR}, one can show that the Fourier-Mukai transform
induces an equivalence of categories
\[
	\FMt_A \colon \Db(\Rmod_A) \to \Db \bigl( \shO_{E(A)} \bigr),
\]
which restricts to an equivalence between the coherent subcategories. Since this fact
will not be used below, we again omit the proof.
\end{note}

\subsection{Compatibility}
\label{subsec:compatible}

Just as $\Rmod_A$-modules interpolate between $\Dmod_A$-modules and quasi-coherent
sheaves on the cotangent bundle $T^{\ast} A$, Definition~\ref{def:FM-R} interpolates
between the Fourier-Mukai transform for $\Dmod_A$-modules and the usual Fourier-Mukai
transform for quasi-coherent sheaves. The purpose of this section is to make that
relationship precise. 

Throughout the discussion, let $\Mmod$ be a coherent $\Dmod_A$-module and $F =
F_{\bullet} \Mmod$ a good filtration of $\Mmod$ by $\OA$-coherent subsheaves. The
graded $\Sym
\shT_A$-module $\gr_{\bullet}^F \Mmod$ is then coherent over $\Sym \shT_A$, and
therefore defines a coherent sheaf on the cotangent bundle $T^{\ast} A$ that we
denote by $\shC(\Mmod, F)$. Now consider the Rees module
\[
	R_F \Mmod = \bigoplus_{k \in \ZZ} F_k \Mmod \cdot z^k 
		\subseteq \Mmod \tensor_{\OA} \OA \lbrack z, z^{-1} \rbrack,
\]
which is a graded $\Rmod_A$-module, coherent over $\Rmod_A$. The associated quasi-coherent
sheaf on $A \times \CC$, which we again denote by the symbol $\ReesM$, is then equivariant
for the natural $\CC^{\ast}$-action on the product.  Moreover, it is easy to see that
the restriction of $\ReesM$ to $A \times \{1\}$ is a $\Dmod_A$-module isomorphic to
$\Mmod$, while the restriction to $A \times \{0\}$ is a $\Sym \shT_A$-module
isomorphic to $\gr_{\bullet}^F \Mmod$.

\begin{proposition} \label{prop:compatible}
Let $\Mmod$ be a coherent $\Dmod_A$-module with good filtration $F_{\bullet} \Mmod$.
Then the Fourier-Mukai transform $\FMt_A(R_F \Mmod)$ has the following properties:
\begin{renumerate}
\item It is equivariant for the natural $\CC^{\ast}$-action on the
vector bundle $E(A)$.
\label{en:FM-R-i}
\item Its restriction to $\Ash = \lambda^{-1}(1)$ is canonically
isomorphic to $\FM_A(\Mmod)$.
\label{en:FM-R-ii}
\item Its restriction to $A \times H^0(A, \Omega_A^1) =
\lambda^{-1}(0)$ is canonically isomorphic to
\[
	\derR (p_{23})_{\ast} \Bigl( p_{12}^{\ast} P \tensor p_1^{\ast} \Omega_A^g
			\tensor p_{13}^{\ast} \shC(\Mmod, F) \Bigr),
\]
where the notation is as in the diagram in \eqref{eq:diag-cotangent} below.
\label{en:FM-R-iii}
\end{renumerate}
\end{proposition}

We begin the proof with two simple lemmas that describe how $E(A)$ and $(\Pg,
\nablag)$ behave under restriction to the fibers of $\lambda \colon E(A) \to \CC$.

\begin{lemma} \label{lem:compatible-1}
We have $\lambda^{-1}(1) = \Ash$, and the restriction of $(\Pg, \nablag)$ to
$A \times \Ash$ is equal to $(\Psh, \nablash)$.
\end{lemma}

\begin{proof}
This follows from the construction of $\Ash$ in \cite{MM}*{Chapter~I}.
\end{proof}

Recall that the cotangent bundle of $A$ satisfies $T^{\ast} A = A \times H^0(A,
\Omega_A^1)$, and consider the following diagram:
\begin{diagram}{2.5em}{2.5em} \label{eq:diag-cotangent}
\matrix[math] (m) {
	A \times H^0(A, \Omega_A^1) & A \times \Ah \times H^0(A, \Omega_A^1)
		& \Ah \times H^0(A, \Omega_A^1) \\
	& A \times \Ah \\
}; 
\path[to] (m-1-2) edge node[above] {$p_{13}$} (m-1-1)
	(m-1-2) edge node[auto] {$p_{12}$} (m-2-2)
	(m-1-2) edge node[auto] {$p_{23}$} (m-1-3);
\end{diagram}

\begin{lemma} \label{lem:compatible-0}
We have $\lambda^{-1}(0) = \Ah \times H^0(A, \Omega_A^1)$, and the restriction of
$(\Pg, \nablag)$ to $A \times \Ah \times H^0(A, \Omega_A^1)$ is equal to the
Higgs bundle
\[
	\bigl( p_{12}^{\ast} P, \, p_{23}^{\ast} \theta_A \bigr),
\]
where $\theta_A$ denotes the tautological holomorphic one-form on $T^{\ast} A$.
\end{lemma}

\begin{proof}
This follows easily from the proof of Lemma~\ref{lem:connection}.
\end{proof}

We can now prove the asserted compatibilities between Definition~\ref{def:FM-R} and
the Fourier-Mukai transforms for $\Dmod_A$-modules and coherent sheaves.

\begin{proof}[Proof of Proposition~\ref{prop:compatible}]
\ref{en:FM-R-i} is true because $R_F \Mmod$ is a graded $\Rmod_A$-module, and because
$\Pg$, $\nablag$, and the relative de Rham complex are obviously
$\CC^{\ast}$-equivariant. \ref{en:FM-R-ii} follows directly from the definition of the
Fourier-Mukai transform, using the base change formula for the morphism $\lambda
\colon E(A) \to \CC$ and Lemma~\ref{lem:compatible-1}.

The proof of \ref{en:FM-R-iii} is a little less obvious, and so we give some details.
By base change, it suffices to show that the restriction of the relative de Rham complex
\[
	\DR_{A \times E(A) / E(A)} \Bigl( (\id \times \lambda)^{\ast} \Mmodt
		\tensor_{\shO_{A \times E(A)}} \bigl( \Pg, \nablag \bigr) \Bigr)
\]
to $A \times \Ah \times H^0(A, \Omega_A^1)$ is a resolution of the
coherent sheaf $p_{12}^{\ast} P \tensor p_1^{\ast} \Omega_A^g \tensor p_{13}^{\ast}
\shC(\Mmod, F)$. After a short computation, one finds that this restriction is
isomorphic to the tensor product of $p_{12}^{\ast} P$ and the pullback, via $p_{13}
\colon A \times \Ah \times H^0(A, \Omega_A^1) \to A \times H^0(A, \Omega_A^1)$, of
the complex
\[
	\biggl\lbrack 
		p_1^{\ast} \Bigl( \gr_{\bullet}^F \Mmod \Bigr) \to 
			p_1^{\ast} \Bigl( \Omega_A^1 \tensor_{\OA} \gr_{\bullet}^F \Mmod \Bigr) 
		\to \dotsb \to 
			p_1^{\ast} \Bigl( \Omega_A^g \tensor_{\OA} \gr_{\bullet}^F \Mmod \Bigr)
	\biggr\rbrack,
\]
placed in degrees $-g, \dotsc, 0$, with differential
$p_1^{\ast} \bigl( \Omega_A^k \tensor \gr_{\bullet}^F \Mmod \bigr) \to
		p_1^{\ast} \bigl( \Omega_A^{k+1} \tensor \gr_{\bullet}^F \Mmod \bigr)$
given by the formula
\[
	\omega \tensor m \mapsto (-1)^{g+k} \left( \omega \wedge \theta_A \tensor m + 
		\sum_{i=1}^g \omega \wedge \omega_i \tensor \partial_i m \right).
\]
But since $\shC(\Mmod, F)$ is the coherent sheaf on $A \times H^0(A, \Omega_A^1)$
corresponding to the $\Sym \shT_A$-module $\gr_{\bullet}^F \Mmod$, said complex
resolves the coherent sheaf $p_1^{\ast} \Omega_A^g \tensor \shC(\Mmod, F)$, and so we
get the desired result.
\end{proof}

\section{Results about cohomology support loci}

\subsection{Cohomology of constructible complexes}
\label{subsec:constructible}

In this section, we describe an analogue of the Fourier-Mukai transform for
constructible complexes on $A$, and use it to prove that cohomology support loci are
algebraic subvarieties of $\Char(A)$. We refer the reader to \cite{HTT}*{Section~4.5}
and to \cite{Dimca}*{Chapter~4} for details about constructible complexes and
perverse sheaves.

As a complex manifold, the abelian variety $A$ may be presented as a quotient $V /
\Lambda$, where $V$ is a complex vector space of dimension $g$, and $\Lambda
\subseteq V$ is a lattice of rank $2g$. Note that $V$ is isomorphic to the tangent
space of $A$ at the unit element, while $\Lambda$ is isomorphic to the fundamental
group $\pi_1(A, 0_A)$. We shall denote by $R = \CC \lbrack \Lambda \rbrack$ the group
ring of $\Lambda$; thus
\[
	R = \bigoplus_{\lambda \in \Lambda} \CC e_{\lambda},
\]
with $e_{\lambda} \cdot e_{\mu} = e_{\lambda + \mu}$. A choice of basis for $\Lambda$
shows that $R$ is isomorphic to the ring of Laurent polynomials in $2g$ variables.
Any character $\rho \colon \Lambda \to \Cst$ extends uniquely to a homomorphism of
$\CC$-algebras
\[
	R \to \CC, \quad e_{\lambda} \mapsto \rho(\lambda),
\]
whose kernel is a maximal ideal $\mmrho$ of $R$; concretely, $\mmrho$ is 
generated by the elements $e_{\lambda} - \rho(\lambda)$ for $\lambda \in \Lambda$.
It is easy to see that any maximal ideal of $R$ is of this form;
this means that $\Char(A)$ is the set of closed points of the scheme $\Spec R$, and
therefore naturally an affine algebraic variety over $\Spec \CC$.

For any finitely generated $R$-module $M$, multiplication by the ring elements
$e_{\lambda}$ determines a natural action of $\Lambda$ on the $\CC$-vector space $M$.
By the well-known correspondence between representations of the fundamental group and
local systems, it thus gives rise to a local system on $A$.

\begin{definition}
For a finitely generated $R$-module $M$, we denote by the symbol $\shL_M$ the
corresponding local system of $\CC$-vector spaces on $A$.
\end{definition}

Since $R$ is commutative, $\shL_M$ is actually a \emph{local system of $R$-modules}. The
most important case of this construction is $\shL_R$, which is a local system of
$R$-modules of rank one; one can show that it is isomorphic to the direct image with
proper support $\pi_! \, \CC_V$ of the constant local system on the covering space $\pi
\colon V \to A$, but we do not need this fact here. The device above allows us
to construct finitely-generated $R$-modules (and hence coherent sheaves on $\Spec R$)
by twisting a constructible complex on $A$ by a local system of the form $\shL_M$,
and pushing forward along the morphism $p \colon A \to \Spec \CC$ to a point.

\begin{proposition}
Let $E \in \Dbc(\CC_A)$. Then for any finitely generated $R$-module $M$, the direct
image $\derR \pl \bigl( E \tensor_{\CC} \shL_M \bigr)$ belongs to $\Dbcoh(R)$.
\end{proposition}

\begin{proof}
Since $E$ is a constructible complex of sheaves of $\CC$-vector spaces, the tensor
product $E \tensor_{\CC} \shL_M$ is a constructible complex of sheaves of
$R$-modules. By \cite{Dimca}*{Corollary~4.1.6}, its direct image is thus an object of
$\Dbcoh(R)$.
\end{proof}

To understand how $\derR \pl \bigl( E \tensor_{\OA} \shL_M \bigr)$ depends on $M$, we
will need the following auxiliary lemma. Recall that a \define{fine sheaf} on a
manifold is a sheaf admitting partitions of unity; such sheaves are acyclic for
direct image functors.

\begin{lemma} \label{lem:functor}
Let $\shF$ be a fine sheaf of $\CC$-vector spaces on $A$. Then the space of global
sections $H^0 \bigl( A, \shF \tensor_{\CC} \shL_R \bigr)$ is a flat $R$-module, and
for every finitely generated $R$-module $M$, one has
\[
	H^0 \bigl( A, \shF \tensor_{\CC} \shL_M \bigr) \simeq
		H^0 \bigl( A, \shF \tensor_{\CC} \shL_R \bigr) \tensor_R M,
\]
functorially in $M$.
\end{lemma}

\begin{proof}
It is easy to see that each sheaf of the form $\shF \tensor_{\CC} \shL_M$ is again a
fine sheaf. Consequently, $M \mapsto H^0 \bigl( A, \shF \tensor_{\CC} \shL_M
\bigr)$ is an exact functor from the category of finitely generated $R$-modules to
the category of $R$-modules. Since the functor also preserves direct sums, the result
follows from the Eilenberg-Watts theorem in homological algebra \cite{Watts}. A
direct proof can be had by resolving $M$ by a bounded complex of free $R$-modules of
finite rank, and using the exactness of the functor.
\end{proof}

\begin{proposition} \label{prop:iso-c}
Let $E \in \Dbc(\CC_A)$. Then for every finitely generated $R$-module $M$, one has an
isomorphism
\[
	\derR \pl \bigl( E \tensor_{\CC} \shL_M \bigr) \simeq
		\derR \pl \bigl( E \tensor_{\CC} \shL_R \bigr) \Ltensor_R M,
\]
functorial in $M$.
\end{proposition}

\begin{proof}
We begin by choosing a bounded complex $(\shF^{\bullet}, d)$ of fine sheaves
quasi-isomorphic to $E$. One way to do this is as follows. By the Riemann-Hilbert
correspondence, $E \simeq \DR_A(\Mmod)$ for some $\Mmod \in \Dbrh(\Dmod_A)$; if
we now let $\shA_A^k$ denote the sheaf of smooth $k$-forms on the complex manifold $A$,
then by the Poincar\'e lemma, 
\[
	\Bigl\lbrack \shA_A^0 \tensor \Mmod \to \shA_A^1 \tensor \Mmod \to \dotsb 
		\to \shA_A^{2g} \tensor \Mmod \Bigr\rbrack,
\]
placed in degrees $-g, \dotsc, g$, is a complex of fine sheaves quasi-isomorphic to $E$.
For any such choice, $\derR \pl \bigl( E \tensor_{\CC} \shL_M \bigr) \in \Dbcoh(R)$
is represented by the bounded complex of $R$-modules with terms
\[
	H^0 \bigl( A, \shF^{\bullet} \tensor_{\CC} \shL_M \bigr) \simeq
		H^0 \bigl( A, \shF^{\bullet} \tensor_{\CC} \shL_R \bigr) \tensor_R M,
\]
and so the assertion follows from Lemma~\ref{lem:functor}.
\end{proof}

Now let $\rho \in \Char(A)$ be an arbitrary character; recall that
$\mmrho$ is the maximal ideal of $R$ generated by the elements $e_{\lambda} -
\rho(\lambda)$, for $\lambda \in \Lambda$. Using the notation introduced above, we
therefore have the alternative description $\CCrho \simeq \shL_{R / \mmrho}$ for the
local system corresponding to $\rho$.

\begin{corollary} \label{cor:fibers}
For any character $\rho \in \Char(A)$, we have 
\[
	\derR \pl \bigl( E \tensor_{\CC} \CCrho \bigr) \simeq
	\derR \pl \bigl( E \tensor_{\CC} \shL_R \bigr) \Ltensor_R R/\mmrho
\]
as objects of $\Dbcoh(\CC)$.
\end{corollary}

\begin{note}
We may thus consider $\derR \pl \bigl( E \tensor_{\CC} \shL_R \bigr) \in \Dbcoh(R)$
as being something like a ``Fourier-Mukai transform'' of the constructible complex $E
\in \Dbc(\CC_A)$. In this setting, however, the transform does not determine the 
original constructible complex: for example, if $E$ is any constructible sheaf whose
support is a finite union of points of $A$, then $\derR \pl \bigl( E \tensor_{\CC}
\shL_R \bigr)$ is a free $R$-module of finite rank.
\end{note}

The results above are sufficient to prove that the cohomology support loci of $E$ are
algebraic subsets of $\Char(A)$.

\begin{theorem} \label{thm:alg-c}
If $E \in \Dbc(\CC_A)$, then each cohomology support locus $S_m^k(E)$ is an
algebraic subset of $\Char(A)$.
\end{theorem}

\begin{proof}
Recall that $\Char(A)$ is the complex manifold associated to the complex algebraic
variety $\Spec R$. Thus $\derR \pl \bigl( E \tensor_{\CC} \shL_R
\bigr) \in \Dbcoh(R)$ determines an object in the bounded derived category of
algebraic coherent sheaves on $\Char(A)$, whose fiber at any closed point $\rho$
computes the hypercohomology of the twist $E \tensor_{\CC} \CCrho$, according to
Corollary~\ref{cor:fibers}.  We conclude that
\[
	S_m^k(E) = \Menge{\rho \in \Char(A)}{\dim \HH^k \Bigl( \derR \pl 		
		\bigl( E \tensor_{\CC} \shL_R \bigr) \Ltensor_R R/\mmrho \Bigr) \geq m},
\]
and by the same argument as in the proof of Proposition~\ref{prop:alg-h}, this
description implies that $S_m^k(E)$ is an algebraic subset of $\Char(A)$.
\end{proof}

\begin{proposition} \label{prop:subfield}
Let $k$ be any subfield of $\CC$. If $E \in \Dbc(k_A)$ is a constructible complex of
sheaves of $k$-vector spaces, then $\derR \pl \bigl( E \tensor_k \shL_R \bigr)$ is
defined over $k$.
\end{proposition}

\begin{proof}
Let $\QQ \lbrack \Lambda \rbrack$ be the group ring with rational coefficients. Then
$\shL_R$ is the complexification of the associated local system of $\QQ$-vector
spaces, and so $\derR \pl \bigl( E \tensor_k \shL_R \bigr)$ is in
an evident manner the complexification of an object of $\Dbcoh \bigl( k \lbrack
\Lambda \rbrack \bigr)$.
\end{proof}

\subsection{Structure theorem}
\label{subsec:structure}

The goal of this section is to prove a fundamental structure theorem for cohomology
support loci of constructible and holonomic complexes. We refer to
\cite{HTT}*{Chapter~3} for details about holonomic $\Dmod$-modules and holonomic
complexes.

\begin{lemma} \label{lem:relationship}
Let $\Mmod \in \Dbh(\Dmod_A)$ be a holonomic complex on $A$. Then we have
\[
	\Phi \bigl( S_m^k(\Mmod) \bigr) = S_m^k \bigl( \DR_A(\Mmod) \bigr).
\]
for every $k,m \in \ZZ$.
\end{lemma}

\begin{proof}
Let $(L, \nabla)$ be a line bundle with integrable connection on $A$. The associated
local system $\ker \nabla$ is a subsheaf of $L$, and so we obtain a morphism of complexes
\begin{equation} \label{eq:morphism}
	\DR_A(\Mmod) \tensor_{\CC} (\ker \nabla) \to 
		\DR_A \bigl( \Mmod \tensor_{\OA} (L, \nabla) \bigr).
\end{equation}
Now $(\ker \nabla) \tensor_{\CC} \OA = L$, and therefore $\Omega_A^k \tensor_{\OA} \Mmod
\tensor_{\CC} (\ker \nabla) = \Omega_A^k \tensor_{\OA} \Mmod \tensor_{\OA} L$. This
shows that the two complexes in \eqref{eq:morphism} are isomorphic to each other, and
gives the desired relation between their hypercohomology groups.
\end{proof}

With the help of the Fourier-Mukai transform, it is easy to show that 
cohomology support loci are algebraic subsets of $\Ash$.

\begin{proposition} \label{prop:alg-h}
If $\Mmod \in \Dbcoh(\Dmod_A)$, then each cohomology support locus $S_m^k(\Mmod)$ is
an algebraic subset of $\Ash$.
\end{proposition}

\begin{proof}
Since $\Ash$ is a quasi-projective algebraic variety, we may represent $\FM_A(\Mmod)$
by a bounded complex $(\shE^{\bullet}, d)$ of locally free sheaves on $\Ash$. Now let
$(L, \nabla)$ be a line bundle with integrable connection, and let $i_{(L, \nabla)}$
denote the inclusion
map. By the base change theorem,
\[
	\derR i_{(L, \nabla)}^{\ast} \FM_A(\Mmod) 
		\simeq \DR_A \bigl( \Mmod \tensor_{\OA} (L, \nabla) \bigr),
\]
and so we have
\[
	S_m^k(\Mmod) = \Menge{(L, \nabla) \in \Ash}%
		{\dim \HH^k \bigl( i_{(L, \nabla)}^{\ast} (\shE^{\bullet}, d) \bigr) \geq m}.
\]
This description shows that $S_m^k(\Mmod)$ is an algebraic subset of $\Ash$, as
claimed.
\end{proof}

We can now prove the structure theorem from the introduction.

\begin{proof}[Proof of Theorem~\ref{thm:h-linear}]
Let $\Mmod \in \Dbh(\Dmod_A)$ be a holonomic complex. Then $\DR_A(\Mmod)$ is
constructible, and we have
\[
	\Phi \bigl( S_m^k(\Mmod) \bigr) = S_m^k \bigl( \DR_A(\Mmod) \bigr)
\]
by Lemma~\ref{lem:relationship}. Proposition~\ref{prop:alg-h} shows that $S_m^k(\Mmod)$
is an algebraic subset of $\Ash$; Theorem~\ref{thm:alg-c} shows that $S_m^k \bigl(
\DR_A(\Mmod) \bigr)$ is an algebraic subset of $\Char(A)$. We conclude from Simpson's
Theorem~\ref{thm:simpson} that both must be finite unions of linear subvarieties of
$\Ash$ and $\Char(A)$, respectively. The assertion about objects of geometric origin
is proved in \subsecref{subsec:geometric} below.
\end{proof}

\begin{proof}[Proof of Corollary~\ref{cor:FM-linear}]
A familiar consequence of the base change theorem is that we have, for every $n \in
\ZZ$, an equality of sets
\begin{equation} \label{eq:base-change}
	\bigcup_{k \geq n} \Supp \shH^k \FM_A(\Mmod) = 
		\bigcup_{k \geq n} S^k(\Mmod).
\end{equation}
Both assertions therefore follow from Theorem~\ref{thm:h-linear}.
\end{proof}

\subsection{Objects of geometric origin}
\label{subsec:geometric}

In this section, we study cohomology support loci for semisimple regular holonomic
$\Dmod_A$-modules of geometric origin, as defined in \cite{BBD}*{6.2.4}. To begin
with, recall the following definition due to Saito \cite{Saito-MM}*{Definition~2.6}.

\begin{definition}
A mixed Hodge module is said to be \define{of geometric origin} if it is 
obtained by applying several of the standard cohomological functors $H^i \fl$,
$H^i f_!$, $H^i \fu$, $H^i f^!$, $\psi_g$, $\phi_{g, 1}$, $\mathbf{D}$, $\boxtimes$,
$\oplus$, $\otimes$, and $\shHom$ to the trivial Hodge structure $\QQ^H$ of weight
zero, and then taking subquotients in the category of mixed Hodge modules.
\end{definition}

One of the results of Saito's theory is that any semisimple perverse sheaf of
geometric origin, in the sense of \cite{BBD}*{6.2.4}, is a direct summand of a
perverse sheaf underlying a mixed Hodge module of geometric origin.
Consequently, any semisimple regular holonomic $\Dmod$-module of geometric origin is
a direct summand of a $\Dmod$-module underlying a mixed Hodge module of
geometric origin.

\begin{theorem} \label{thm:geom-origin}
Let $\Mmod$ be a semisimple regular holonomic $\Dmod_A$-module of geometric origin.
Then each cohomology support locus $S_m^k(\Mmod)$ is a finite union of arithmetic
linear subvarieties of $\Ash$.
\end{theorem}

We introduce some notation that will be used during the proof. For any field
automorphism $\sigma \in \Aut(\CC/\QQ)$, we obtain from $A$ a new complex abelian
variety $\Asig$. Likewise, an algebraic line bundle $(L, \nabla)$ with integrable
connection on $A$ gives rise to $(\Lsig, \nablasig)$ on $\Asig$, and so we have a
natural map
\[
	\csig \colon \Ash \to \Asigsh.
\]
Now recall the following notion, due in a slightly different form to Simpson, who
modeled it on Deligne's definition of absolute Hodge classes.

\begin{definition} \label{def:absolute}
A closed subset $Z \subseteq \Ash$ is said to be \define{absolute closed} if, for
every field automorphism $\sigma \in \Aut(\CC/\QQ)$, the set
\[
	F \bigl( \csig(Z) \bigr) \in \Char(\Asig)
\]
is closed and defined over $\QQb$.
\end{definition}

The following theorem about absolute closed subsets is also due to Simpson.

\begin{theorem}[Simpson]
An absolute closed subset of $\Ash$ is a finite union of arithmetic linear
subvarieties.
\end{theorem}

\begin{proof}
Simpson's definition \cite{Simpson}*{p.~376} of absolute closed sets actually
contains several additional conditions (related to the space of Higgs bundles); but
as he explains, a strengthening of \cite{Simpson}*{Theorem~3.1}, added in proof,
makes these conditions unnecessary. In fact, the proof of
\cite{Simpson}*{Theorem~6.1} goes through unchanged with only the
assumptions in Definition~\ref{def:absolute}.
\end{proof}

With the help of Simpson's result, the proof of Theorem~\ref{thm:geom-origin} is
straightforward. We first establish the following lemma.

\begin{lemma} \label{lem:MHM}
Let $M \in \MHM(A)$ be a mixed Hodge module, with underlying filtered
$\Dmod_A$-module $(\Mmod, F)$. Then the cohomology support loci of the
perverse sheaf $\DR_A(\Mmod)$ are algebraic subsets of $\Char(A)$ that are defined
over $\QQb$.
\end{lemma}

\begin{proof}
By definition, a mixed Hodge module has an underlying perverse sheaf $\rat M$ with
coefficients in $\QQ$, and $\DR_A(\Mmod) \simeq (\rat M) \tensor_{\QQ} \CC$. By
Proposition~\ref{prop:subfield}, it follows that $\derR \pl \bigl( \DR_A(\Mmod)
\tensor_{\CC} \shL_R \bigr) \in \Dbcoh(R)$ is obtained by extension of scalars from
an object of $\Dbcoh \bigl( \QQ \lbrack \Lambda \rbrack \bigr)$.
The assertion about cohomology support loci now follows easily from
Corollary~\ref{cor:fibers}.
\end{proof}

\begin{note}
The same result is true for any holonomic $\Dmod_A$-module with
$\QQb$-structure; that is to say, for any holonomic $\Dmod_A$-module whose de Rham
complex is the complexification of a perverse sheaf with coefficients in $\QQb$.
This is what Mochizuki calls a ``pre-Betti structure'' in \cite{Mochizuki}.
\end{note}

\begin{lemma} \label{lem:subquotient}
Let $E \in \Dbc(\QQb_A)$ be a perverse sheaf with coefficients in $\QQb$. Any
simple subquotient of $E \tensor_{\QQb} \CC$ is the complexification of a simple
subquotient of $E$.
\end{lemma}

\begin{proof}
We only have to show that if $E \in \Dbc(\QQb_A)$ is a simple perverse sheaf, then $E
\tensor_{\QQb} \CC \in \Dbc(\CC_A)$ is also simple. By the classification of simple
perverse sheaves, there is an irreducible locally closed subvariety $U \subseteq A$,
and an irreducible representation $\rho \colon \pi_1(U) \to \GL_n(\QQb)$, such that
$E$ is the intermediate extension of the local system associated to $\rho$. Since
$\QQb$ is algebraically closed, $\rho$ remains irreducible over $\CC$,
proving that $E \tensor_{\QQb} \CC$ is still simple.
\end{proof}

\begin{proof}[Proof of Theorem~\ref{thm:geom-origin}]
We first show that this holds when $\Mmod$ underlies a mixed Hodge module $M$
obtained by iterating the standard cohomological functors (but without taking
subquotients). Fix two integers $k, m$, and set $Z = S_m^k(\Mmod)$. In light of
Lemma~\ref{lem:MHM}, it suffices to prove that each set $\csig(Z)$ is equal to
$S_m^k(\Mmodsig)$ for some polarizable Hodge module $\Msig \in \MHM(\Asig)$. But
since $M$ is of geometric origin, this is obviously the case; indeed, we can obtain
$\Msig$ by simply applying $\sigma$ to the finitely many algebraic varieties and
morphisms involved in the construction of $M$. 

Now suppose that $\Mmod$ is an arbitrary semisimple regular holonomic
$\Dmod_A$-module of geometric origin. Then $\Mmod$ is a direct sum of simple
subquotients of $\Dmod_A$-modules underlying mixed Hodge modules of geometric origin.
By the same argument as before, it suffices to show that $\Mmodsig$ is defined over
$\QQb$ for any $\sigma \in \Aut(\CC/\QQ)$. Now the perverse sheaf $\DR_A(\Mmodsig)$
is again a direct sum of simple subquotients of perverse sheaves underlying mixed Hodge
modules; by Lemma~\ref{lem:subquotient}, it is therefore the complexification of a
perverse sheaf with coefficients in $\QQb$. We then conclude the proof as above.
\end{proof}

\subsection{Perverse coherent sheaves}
\label{subsec:perverse}

Let $X$ be a smooth complex algebraic variety. In this section, we recall the
construction of perverse t-structures on the bounded derived category $\Dbcoh(\OX)$
of algebraic coherent sheaves, following \cite{Kashiwara}.  For a (possibly
non-closed) point $x$ of the scheme $X$, we denote the residue field at the point by
$\kappa(x)$, the inclusion morphism by $i_x \colon \Spec \kappa(x) \into X$,
and the codimension of the closed subvariety $\overline{ \{x\} }$ by $\codim(x) =
\dim \shO_{X,x}$.

\begin{definition} 
A \define{supporting function} on $X$ is a function $p \colon X \to \ZZ$ from the
underlying topological space of the scheme $X$ to the set of integers, with the property
that $p(y) \geq p(x)$ whenever $y \in \overline{ \{x\} }$. 
\end{definition}
Given a supporting function, Kashiwara defines two families of subcategories
\begin{align*}
	\pDtcoh{p}{\leq k}(\OX) &= 
		\menge{E \in \Dbcoh(\OX)}%
			{\text{$\derL i_x^{\ast} E \in \Dtcoh{\leq k+p(x)} %
			\bigl( \kappa(x) \bigr)$ for all $x \in X$}}, \\
	\pDtcoh{p}{\geq k}(\OX) &= 
		\menge{E \in \Dbcoh(\OX)}%
			{\text{$\derR i_x^! E \in \Dtcoh{\geq k+p(x)} %
			\bigl( \kappa(x) \bigr)$ for all $x \in X$}}.
\end{align*}
The following fundamental result is proved in \cite{Kashiwara}*{Theorem~5.9} and,
based on an idea of Deligne, in \cite{AB}*{Theorem 3.10}.

\begin{theorem}[Kashiwara] \label{thm:kashiwara}
The above subcategories define a bounded t-structure on $\Dbcoh(\OX)$ if, and only
if, the supporting function has the property that 
\[	
	p(y) - p(x) \leq \codim(y) - \codim(x)
\]
for every pair of (possibly non-closed) points $x,y \in X$ with $y \in \overline{
\{x\} }$.
\end{theorem}

For example, $p=0$ corresponds to the standard t-structure on
$\Dbcoh(\OX)$. An equivalent way of putting the condition in
Theorem~\ref{thm:kashiwara} is that the dual function $\hat{p}(x) = \codim(x) - p(x)$
should again be a supporting function. If that is the case, one has the identities
\begin{align*}
	\pDtcoh{\hat{p}}{\leq k}(\OX) &=
		 \derR \shHom \Bigl( \pDtcoh{p}{\geq -k}(\OX), \OX \Bigr) \\
	\pDtcoh{\hat{p}}{\geq k}(\OX) &=
		 \derR \shHom \Bigl( \pDtcoh{p}{\leq -k}(\OX), \OX \Bigr),
\end{align*}
which means that the duality functor $\derR \shHom(\argbl, \OX)$ exchanges the
two perverse t-structures defined by $p$ and $\hat{p}$. 

\begin{definition}
The heart of the t-structure defined by $p$ is denoted
\[
	\pCoh{p}(\OX) = \pDtcoh{p}{\leq 0}(\OX) \cap \pDtcoh{p}{\geq 0}(\OX),
\]
and is called the abelian category of \define{$p$-perverse coherent sheaves}. 
\end{definition}

We are interested in a special cases of Kashiwara's result, namely that the set of
objects $E \in \Dbcoh(\OX)$ with $\codim \Supp \shH^i(E) \geq 2i$ for all $i \geq 0$ is
part of a t-structure on $\Dbcoh(\OX)$. To formalize this idea, define a
function 
\[
	m \colon X \to \ZZ, \quad 
		m(x) = \left\lfloor \tfrac{1}{2} \codim(x) \right\rfloor.
\]
It is easily verified that both $m$ and the dual function
\[
	\hat{m} \colon X \to \ZZ, \quad
		\hat{m}(x) = \left\lceil \tfrac{1}{2} \codim(x) \right\rceil
\]
are supporting functions. As a consequence of Theorem~\ref{thm:kashiwara}, $m$ defines a
bounded t-structure on $\Dbcoh(\OX)$; objects of the heart $\mCoh(\OX)$ will be
called \define{$m$-perverse coherent sheaves}.

\begin{note}
We use this letter because $m$ and $\hat{m}$ are as close as one can get to ``middle
perversity''. There is of course no actual middle perversity for coherent sheaves,
because the equality $p = \hat{p}$ cannot hold unless $X$ is a point.
\end{note}

The next lemma follows easily from \cite{Kashiwara}*{Lemma~5.5}.
\begin{lemma} \label{lem:m-structure}
The perverse t-structures defined by $m$ and $\hat{m}$ satisfy
\begin{align*}
	\mDtcoh{\leq k}(\OX) &= \menge{E \in \Dbcoh(X)}%
		{\text{$\codim \Supp \shH^i(E) \geq 2(i-k)$ for all $i \in \ZZ$}} \\
	\pDtcoh{\hat{m}}{\leq k}(\OX) &= \menge{E \in \Dbcoh(X)}%
		{\text{$\codim \Supp \shH^i(E) \geq 2(i-k)-1$ for all $i \in \ZZ$}}.
\end{align*}
By duality, this also describes the subcategories with $\geq k$.
\end{lemma}

Consequently, an object $E \in \Dbcoh(\OX)$ is an $m$-perverse coherent sheaf
precisely when $\codim \Supp \shH^i(E) \geq 2i$ and $\codim \Supp R^i \shHom(E, \OX)
\geq 2i-1$ for every integer $i \geq 0$. This shows one more time that the category
of $m$-perverse coherent sheaves is not preserved by the duality functor $\derR
\shHom(\argbl, \OX)$.

\begin{lemma} \label{lem:m-geq}
If $E \in \mDtcoh{\geq 0}(\OX)$, then $E \in \Dtcoh{\geq 0}(\OX)$.
\end{lemma}

\begin{proof}
This is obvious from the fact that $m \geq 0$. 
\end{proof}

When it happens that both $E$ and $\derR \shHom(E, \OX)$ are $m$-perverse coherent
sheaves, $E$ has surprisingly good properties.

\begin{proposition} \label{prop:surprise}
If $E \in \mDtcoh{\leq 0}(\OX)$ satisfies $\derR \shHom(E, \OX) \in \mDtcoh{\leq
0}(\OX)$, then it has the following properties:
\begin{renumerate}
\item Both $E$ and $\derR \shHom(E, \OX)$ belong to $\mCoh(\OX)$.
\item Let $r \geq 0$ be the least integer with $\shH^r(E) \neq 0$; then $\codim
\Supp \shH^r(E) = 2r$.
\end{renumerate}
\end{proposition}

\begin{proof}
The first assertion follows directly from Lemma~\ref{lem:m-structure}. To prove the
second assertion, note that we have $\codim \Supp \shH^r(E) \geq 2r$. It therefore
suffices to show that if $\shH^i(E) = 0$ for $i < r$, and $\codim \Supp \shH^r(E) >
2r$, then $\shH^r(E) = 0$. Under these assumptions, we have
\[
	\codim \Supp \shH^i(E) \geq \max(2i,2r+1) \geq i+r+1, 
\]
and therefore $\derR \shHom(E, \OX) \in \Dtcoh{\geq r+1}(\OX)$ by
\cite{Kashiwara}*{Proposition~4.3}. The same argument, applied to $\derR \shHom(E,
\OX)$, now shows that $E \in \Dtcoh{\geq r+1}(\OX)$.
\end{proof}

\subsection{Codimension bounds}
\label{subsec:codimension}

In this section, we show that the standard t-structure on $\Dbh(\Dmod_A)$
corresponds, under the Fourier-Mukai transform $\FM_A$, to the $m$-perverse t-structure.

\begin{theorem} \label{thm:t-structure}
Let $\Mmod \in \Dbh(\Dmod_A)$ be a holonomic complex on $A$. Then one has
\begin{align*}
	\Mmod \in \Dth{\leq k}(\Dmod_A) \quad &\Longleftrightarrow \quad 
		\FM_A(\Mmod) \in \mDtcoh{\leq k}(\OAsh), \\
	\Mmod \in \Dth{\geq k}(\Dmod_A) \quad &\Longleftrightarrow \quad 
		\FM_A(\Mmod) \in \mDtcoh{\geq k}(\OAsh).
\end{align*}
\end{theorem}

The first step of the proof consists in the following ``generic vanishing theorem''
for holonomic $\Dmod_A$-modules. In the regular case, this result
is due to Kr\"amer and Weissauer \cite{KW}*{Theorem~2}, whose proof relies on the
(difficult) recent solution of Kashiwara's conjecture for semisimple perverse sheaves. By
contrast, our proof is completely elementary.

\begin{proposition} \label{prop:KW}
Let $\Mmod$ be a holonomic $\Dmod_A$-module. Then for every $i > 0$, the
support of the coherent sheaf $\shH^i \FM_A(\Mmod)$ is a proper subset of $\Ash$.
\end{proposition}

\begin{proof}
Let $F_{\bullet} \Mmod$ be a good filtration by $\OA$-coherent subsheaves; this
exists by \cite{HTT}*{Theorem~2.3.1}. As in \subsecref{subsec:compatible}, we consider the
associated coherent $\Rmod_A$-module $R_F \Mmod$ defined by the Rees construction,
and its Fourier-Mukai transform
\[
	\FMt_A(R_F \Mmod) \in \Db \bigl( \shO_{E(A)} \bigr).
\]
By Proposition~\ref{prop:compatible}, $\FMt_A(R_F \Mmod)$ is equivariant for the
$\CC^{\ast}$-action on $E(A)$, and its restriction to $\lambda^{-1}(1) = \Ash$ is
isomorphic to $\FM_A(\Mmod)$. It is therefore sufficient to prove that the
restriction of $\FMt_A(R_F \Mmod)$ to $\lambda^{-1}(0) = A \times H^0(A, \Omega_A^1)$
has the asserted property. By Proposition~\ref{prop:compatible}, this restriction is
isomorphic to 
\begin{equation} \label{eq:formula}
	\derR (p_{23})_{\ast} \Bigl( p_{12}^{\ast} P \tensor 
		p_{13}^{\ast} \shC(\Mmod, F) \tensor p_1^{\ast} \Omega_A^g \Bigr),
\end{equation}
where the notation is as in the following diagram:
\begin{diagram*}{2.5em}{2.5em}
\matrix[math] (m) {
	A \times H^0(A, \Omega_A^1) & A \times \Ah \times H^0(A, \Omega_A^1)
		& \Ah \times H^0(A, \Omega_A^1) \\
	& A \times \Ah \\
}; 
\path[to] (m-1-2) edge node[above] {$p_{13}$} (m-1-1)
	(m-1-2) edge node[auto] {$p_{12}$} (m-2-2)
	(m-1-2) edge node[auto] {$p_{23}$} (m-1-3);
\end{diagram*}
But $\Mmod$ is holonomic, and so each irreducible component of the support of
$\shC(\Mmod, F)$ has dimension $g$. Thus the restriction of $p_{23}$
to the support of $p_{13}^{\ast} \shC(\Mmod, F)$ is generically finite over $\Ah
\times H^0(A, \Omega_A^1)$, which implies that the support of the higher direct image
sheaves in \eqref{eq:formula} is a proper subset of $\Ah \times H^0(A, \Omega_A^1)$.
\end{proof}

Together with the structure theory for cohomology support loci and basic properties
of the Fourier-Mukai transform, this result now allows us to prove the first
equivalence asserted in Theorem~\ref{thm:t-structure}.

\begin{lemma} \label{lem:t-structure1}
For any $\Mmod \in \Dth{\leq k}(\Dmod_A)$, one has $\FM_A(\Mmod) \in \mDtcoh{\leq k}(\OAsh)$.
\end{lemma}

\begin{proof}
The proof is by induction on $\dim A$, the statement being obviously true when $A$ is
a point. Since $\FM_A$ is triangulated, it suffices to prove the statement for $k=0$.
According to Lemma~\ref{lem:m-structure}, what we then need to show is the 
following: for any holonomic complex $\Mmod \in \Dth{\leq 0}(\Dmod_A)$ concentrated in
nonpositive degrees, the Fourier-Mukai transform $\FM_A(\Mmod)$ satisfies, for every
$\ell \geq 1$, the inequality
\[
	\codim \Supp \shH^{\ell} \FM_A(\Mmod) \geq 2 \ell.
\]
Let $Z$ be any irreducible component of the support of $\shH^{\ell} \FM_A(\Mmod)$,
for some $\ell \geq 1$. By \eqref{eq:base-change} and descending induction on $\ell$,
we may assume that $Z$ is also an irreducible component of $S^{\ell}(\Mmod)$;
according to Theorem~\ref{thm:h-linear}, $Z$ is therefore a linear
subvariety of $\Ash$, and hence of the form $Z = t_{(L, \nabla)}(\im \fsh)$ for a
surjective morphism $f \colon A \to B$ and a suitable point $(L, \nabla) \in \Ash$.
Furthermore, Proposition~\ref{prop:KW} shows that $\codim Z > 0$, and therefore $\dim B
< \dim A$. Setting $r = \dim A - \dim B > 0$, we thus have $\codim Z = 2r$.

Using the properties of the Fourier-Mukai transform listed in
Theorem~\ref{thm:Laumon}, we find that the pullback of $\FM_A(\Mmod)$ to the
subvariety $Z$ is isomorphic to
\[
	\derL (\fsh)^{\ast} \derL t_{(L, \nabla)}^{\ast} \FM_A(\Mmod) \\
		\simeq \FM_B \Bigl( \fp \bigl( \Mmod \tensor_{\OA} (L, \nabla) \bigr) \Bigr)
		\in \Dbcoh(\OBsh).
\]
From the definition of the direct image functor $\fp$, it is clear that $\fp \bigl(
\Mmod \tensor_{\OA} (L, \nabla) \bigr)$ belongs to the subcategory $\Dth{\leq
r}(\Dmod_B)$. The inductive assumption now allows us to
conclude that the restriction of $\FM_A(\Mmod)$ to $Z$ lies in the subcategory $\mDtcoh{\leq
r}(\shO_Z)$. But $Z$ is an irreducible component of $\Supp \shH^{\ell}
\FM_A(\Mmod)$; it follows that $\ell \leq r$, and consequently $\codim Z \geq 2\ell$,
as asserted.
\end{proof}

\begin{lemma} \label{lem:t-structure2}
Let $\Mmod \in \Dbh(\Dmod_A)$ be a holonomic complex. 
If its Fourier-Mukai transform satisfies $\FM_A(\Mmod) \in \Dtcoh{\leq k}(\OAsh)$,
then $\Mmod \in \Dth{\leq k}(\Dmod_A)$.
\end{lemma}

\begin{proof}
It suffices to prove this for $k=0$. By \cite{Laumon}*{Th\'eor\`eme~3.2.1}, we 
can recover $\Mmod$ (up to canonical isomorphism) from its Fourier-Mukai transform as
\[
	\Mmod = \langle -1_A \rangle^{\ast} \derR (p_1)_{\ast} \Bigl(
		\Psh \tensor_{\shO_{A \times \Ash}} p_2^{\ast} \FM_A(\Mmod) \Bigr) \decal{g},
\]
where $p_1 \colon A \times \Ash \to A$ and $p_2 \colon A \times \Ash \to \Ash$ are
the two projections. If we forget about the $\Dmod_A$-module structure and only
consider $\Mmod$ as a complex of quasi-coherent sheaves of $\OA$-modules, we can use
the fact that $\pi \colon \Ash \to A$ is affine to obtain
\[
	\Mmod = \langle -1_A \rangle^{\ast} \derR (p_1)_{\ast} \Bigl(
		 P \tensor_{\shO_{A \times \Ah}} p_2^{\ast} \, \pil \FM_A(\Mmod) \Bigr) \decal{g},
\]
where now $p_1 \colon A \times \Ah \to A$ and $p_2 \colon A \times \Ah \to \Ah$.
By virtue of \eqref{eq:base-change} and Theorem~\ref{thm:h-linear}, each irreducible
component of $\Supp \shH^{\ell} \FM_A(\Mmod)$ is contained in a linear subvariety of
codimension at least $2\ell$; consequently, each irreducible component of $\Supp \pil
\shH^{\ell} \FM_A(\Mmod)$ still has codimension at least $\ell$. From this, it is
easy to see that $\shH^i \Mmod = 0$ for $i > 0$, and hence that $\Mmod \in \Dth{\leq
0}(\Dmod_A)$.
\end{proof}

\begin{proof}[Proof of Theorem~\ref{thm:t-structure}]
The first equivalence is proved in Lemma~\ref{lem:t-structure1} and
Lemma \ref{lem:t-structure2} above. The second equivalence follows from this by duality,
using the compatibility of the Fourier-Mukai transform with the duality functors for
$\Dmod_A$-modules and $\shO_{\Ash}$-modules (see Theorem~\ref{thm:Laumon}).
\end{proof}

\subsection{Proofs for constructible complexes}
\label{subsec:proofs}

For the convenience of the reader, we collect in this section the proofs for the
results announced in \subsecref{subsec:results-c}. We begin with the structure of
the cohomology support loci
\[
	S_m^k(E) = \Menge{\rho \in \Char(A)}{\dim \HH^k \bigl( A, E \tensor_{\CC}
		\CCrho \bigr) \geq m}
\]
of a constructible complex $E \in \Dbc(\CC_A)$.

\begin{proof}[Proof of Theorem~\ref{thm:c-linear}]
To prove \ref{en:c-linear-a}, we use the Riemann-Hilbert correspondence to find
a regular holonomic complex $\Mmod \in \Dbrh(\Dmod_A)$ such that $\DR_A(\Mmod) \simeq
E$. Since $S_m^k(E) = \Phi \bigl( S_m^k(\Mmod) \bigr)$ by
Lemma~\ref{lem:relationship}, the assertion follows from Theorem~\ref{thm:h-linear}.
The statement in \ref{en:c-linear-b} may be deduced from Theorem~\ref{thm:geom-origin} by a
similar argument.
\end{proof}

Next comes the description of the perverse t-structure on $\Dbc(\CC_A)$ in terms of the
codimension of the loci $S^k(E) = S_1^k(E)$. 

\begin{proof}[Proof of Theorem~\ref{thm:c-t}]
Let $\Mmod \in \Dbrh(\Dmod_A)$ be a regular holonomic complex such that $\DR_A(\Mmod)
\simeq E$. Since $\Phi \bigl( S^k(\Mmod) \bigr) = S^k(E)$, the assertion in
\ref{en:c-t-1} is a consequence of Theorem~\ref{thm:t-structure}. To deduce 
\ref{en:c-t-2}, let $\DA \colon \Dbc(\CC_A) \to \Dbc(\CC_A)$ be the
Verdier duality functor. We then have
\[
	S_m^k(E) = \langle -1_{\Char(A)} \rangle \, S_m^{-k}(\DA E)
\]
by Verdier duality. Since $E \in \piDtc{\geq 0}(\CC_A)$ iff $\DA E \in \piDtc{\leq
0}(\CC_A)$, the assertion now follows from \ref{en:c-t-1}. Finally,
\ref{en:c-t-3} is clear from the definition of perverse sheaves as the
heart of the perverse t-structure on $\Dbc(\CC_A)$.
\end{proof}

Lastly, we give the proof of the structure theorem for simple perverse sheaves with
Euler characteristic equal to zero.

\begin{proof}[Proof of Theorem~\ref{thm:c-simple}]
This again follows from the Riemann-Hilbert correspondence and the analogous 
result for simple holonomic $\Dmod_A$-modules in Corollary~\ref{cor:h-simple}.
\end{proof}

\section{The structure of Fourier-Mukai transforms}

\subsection{Simple objects}

In this section, we prove a structure theorem for the Fourier-Mukai transform of
a simple holonomic $\Dmod_A$-module. 

\begin{theorem} \label{thm:simple-support}
Let $\Mmod$ be a simple holonomic $\Dmod_A$-module, and let $r \geq 0$ be the least
integer such that $\shH^r \bigl( \FM_A(\Mmod) \bigr) \neq 0$. Then there is an abelian
variety $B$ of dimension $\dim B = \dim A - r$, a surjective morphism $f \colon A \to
B$ with connected fibers, and a simple holonomic $\Dmod_B$-module $\Nmod$, such that
\[	
	\Mmod \simeq \fu \Nmod \tensor_{\OA} (L, \nabla)
\]
for a suitable point $(L, \nabla) \in \Ash$. Moreover, we have $\Supp \shH^0 \bigl(
\FM_B(\Nmod) \bigr) = \Bsh$.
\end{theorem}

This result clearly implies Theorem~\ref{thm:h-simple} from the introduction. Here is
the proof of the corollary about simple holonomic $\Dmod_A$-modules with Euler
characteristic zero.

\begin{proof}[Proof of Corollary~\ref{cor:h-simple}]
Let $(L, \nabla) \in \Ash$ be a generic point. Since
\[
	0 = \chi(A, \Mmod) = \chi \bigl( A, \Mmod \tensor_{\OA} (L, \nabla) \bigr)
		= \dim \HH^0 \Bigl( A, \DR_A \bigl( \Mmod \tensor_{\OA} (L, \nabla) \bigr) \Bigr),
\]
we find that the support of $\shH^0 \FM_A(\Mmod)$ is a proper subset of $\Ash$. Since
both $\FM_A(\Mmod)$ and the dual complex belong to $\mDtcoh{\leq 0}(\OAsh)$ by
Theorem~\ref{thm:t-structure}, we conclude from Proposition~\ref{prop:surprise} that
$\shH^0 \FM_A(\Mmod) = 0$. Now it only remains to apply
Theorem~\ref{thm:simple-support}.
\end{proof}

For the proof of Theorem~\ref{thm:simple-support}, we need two small lemmas. The
first describes the inverse image of a simple holonomic $\Dmod$-module.

\begin{lemma} \label{lem:inverse-simple}
Let $f \colon A \to B$ be a surjective morphism of abelian varieties, with connected
fibers. If $\Nmod$ is a simple holonomic $\Dmod_B$-module, then $\fu \Nmod$ is a
simple holonomic $\Dmod_A$-module.
\end{lemma}

\begin{proof}
Since $f$ is smooth, $\fu \Nmod = \OA \tensor_{f^{-1} \shO_B} f^{-1} \Nmod$ is a
holonomic $\Dmod_A$-module. By the classification of simple holonomic
$\Dmod$-modules \cite{HTT}*{Theorem~3.4.2}, there is a locally closed subvariety $U
\subseteq B$, and an irreducible representation $\rho \colon \pi_1(U) \to \GL(V)$,
such that $\Nmod$ is the minimal extension of the integrable connection on $U$
associated to $\rho$.  Now $f$ has connected fibers, and so the map on fundamental
groups
\[
	\fl \colon \pi_1 \bigl( f^{-1}(U) \bigr) \to \pi_1(U)
\]
is surjective. Clearly, the pullback $\fu \Nmod$ is equal, over $f^{-1}(U)$, to the
integrable connection associated to the representation $\rho \circ \fl \colon \pi_1
\bigl( f^{-1}(U) \bigr) \to \pi_1(U) \to \GL(V)$. This representation is still 
irreducible because $\fl$ is surjective; to conclude the proof, we shall argue that
$\fu \Nmod$ is the minimal extension.

By \cite{HTT}*{Theorem~3.4.2}, it suffices to show that $\fu \Nmod$ has no
submodules or quotient modules that are supported outside of $f^{-1}(U)$. Suppose
that $\Mmod \into \fu \Nmod$ is such a submodule. We have $\fsi \Nmod = \fu \Nmod
\decal{r}$, where $r = \dim A - \dim B$; by adjunction, the morphism $\Mmod \into \fu
\Nmod$ corresponds to a morphism $\fp \Mmod \decal{r} \to \Nmod$, which factors
uniquely as
\[
	\fp \Mmod \decal{r} \to \shH^r \fp \Mmod \to \Nmod.
\]
Since $\shH^r \fp \Mmod$ is supported outside of $U$, this morphism must be zero;
consequently, $\Mmod = 0$. A similar result for quotient modules can be derived by
applying the duality functor, using \cite{HTT}*{Theorem~2.7.1}. This shows that $\fu
\Nmod$ is the minimal extension of a simple integrable connection, and hence simple.
\end{proof}

The second lemma deals with restriction to an irreducible component of the support of
a complex.

\begin{lemma} \label{lem:support}
Let $X$ be a scheme, and let $E \in \Dbcoh(\OX)$. Suppose that $Z$ is an irreducible
component of the support of $\shH^r(E)$, but not of any $\shH^i(E)$ with $i > r$. Let
$i \colon Z \into X$ be the inclusion. Then the morphism
\[
	\shH^r(E) \to \shH^r \bigl( \derR \il \derL \iu E \bigr)
\]
induced by adjunction is nonzero at the generic point of $Z$.
\end{lemma}

\begin{proof}
After localizing at the generic point of $Z$, we may assume that $X = \Spec R$ for a
local ring $(R, \mm)$, and that $E \in \Dbcoh(R)$ is represented by a complex
\begin{diagram*}{1.5em}{1.5em}
\matrix[math] (m) { \dotsb & E^{r-2} & E^{r-1} & E^r \\ };
\path[to] (m-1-1) edge (m-1-2);
\path[to] (m-1-2) edge node[above] {$d$} (m-1-3);
\path[to] (m-1-3) edge node[above] {$d$} (m-1-4);
\end{diagram*}
of finitely generated free $R$-modules. Set $M = \shH^r(E) = E^r / d E^{r-1}$, which
is a finitely generated $R$-module with $M \neq 0$. Then $\shH^r \bigl( \derR \il \derL
\iu E \bigr) \simeq M / \mm M$, and the morphism $M \to M / \mm M$ is
nonzero by Nakayama's lemma.
\end{proof}

We can now prove our structure theorem for simple holonomic $\Dmod_A$-modules.

\begin{proof}[Proof of Theorem~\ref{thm:simple-support}] 
Let $E = \FM_A(\Mmod) \in \Dbcoh(\OAsh)$. Theorem~\ref{thm:t-structure} shows that $E
\in \mCoh(\OAsh)$; by duality, it follows that $\derR \shHom(E, \OAsh) \in
\mCoh(\OAsh)$, too. According to Proposition~\ref{prop:surprise}, the codimension of
the support of $\shH^r(E)$ is therefore equal to $2r$; moreover, each irreducible
component of $\Supp \shH^r(E)$ of codimension $2r$ is also an irreducible component
of $S^r(\Mmod)$ by \eqref{eq:base-change}. After tensoring $\Mmod$ by a suitable line
bundle with integrable connection, we may assume that one irreducible component of
the support of $\shH^r(E)$ is equal to $\im \fsh$, for a surjective morphism of
abelian varieties $f \colon A \to B$ with connected fibers and $\dim B = \dim A - r$.

To produce the required simple $\Dmod_B$-module, consider the direct image $\fp
\Mmod$, which belongs to $\Dth{\leq r}(\Dmod_B)$. We have a distinguished triangle
\[
	\tau_{\leq r-1}(\fp \Mmod) \to \fp \Mmod \to \shH^r(\fp \Mmod) \decal{-r} \to 
		\dotsb
\]
in $\Dbh(\Dmod_B)$, and hence also a distinguished triangle
\begin{equation} \label{eq:triangle}
	\fsi \tau_{\leq r-1}(\fp \Mmod) \to \fsi \fp \Mmod \to 
		\fsi \shH^r(\fp \Mmod) \decal{-r} \to \dotsb
\end{equation}
in $\Dbh(\Dmod_A)$. Since $f$ is smooth, $\fsi \shH^r(\fp \Mmod) \decal{-r} = \fu
\shH^r(\fp \Mmod)$ is a single holonomic $\Dmod_A$-module. Let $\alpha \colon \Mmod
\to \fsi \fp \Mmod$ be the adjunction morphism. 

Now we observe that the induced morphism $\Mmod \to \fu \shH^r(\fp \Mmod)$ must be 
nonzero. Indeed, suppose to the contrary that the morphism was zero.  Then $\alpha$
factors as
\[
	\Mmod \to \fsi \tau_{\leq r-1}(\fp \Mmod) \to \fsi \fp \Mmod.
\]
If we apply the Fourier-Mukai transform to this factorization, and use the properties
in Theorem~\ref{thm:Laumon}, we obtain
\[
	E \to \derR \fsh_{\ast} \FM_B \bigl( \tau_{\leq r-1}(\fp \Mmod) \bigr) \to 
		\derR \fsh_{\ast} \derL(\fsh)^{\ast} E,
\]
which is a factorization of the adjunction morphism for the closed embedding $\fsh$.
In particular, we then have
\[
	\shH^r(E) \to \shH^r \Bigl( \derR \fsh_{\ast} \FM_B \bigl( \tau_{\leq r-1}(\fp
		\Mmod) \bigr) \Bigr) 
		\to \shH^r \bigl( \derR \fsh_{\ast} \derL(\fsh)^{\ast} E \bigr);
\]
but because the coherent sheaf in the middle is supported in a subset of $\im \fsh$ of
codimension at least two, this contradicts Lemma~\ref{lem:support}. Therefore, $\Mmod
\to \fu \shH^r(\fp \Mmod)$ is indeed nonzero. 

Being a holonomic $\Dmod_B$-module, $\shH^r(\fp \Mmod)$ admits a finite filtration
with simple quotients; consequently, we can find a simple holonomic $\Dmod_B$-module
$\Nmod$ and a nonzero morphism $\Mmod \to \fu \Nmod$. Since $\Mmod$ is simple, and
$\fu \Nmod$ is also simple by Lemma~\ref{lem:inverse-simple}, the morphism must be an
isomorphism, and so $\Mmod \simeq \fu \Nmod$.

To prove the final assertion, note that $\fu \Nmod = \fsi \Nmod \decal{-r}$; on 
account of Theorem~\ref{thm:Laumon}, we therefore have
\[
	\FM_A(\Mmod) \simeq \FM_A(\fu \Nmod) \simeq 
		\derR \fsh_{\ast} \FM_B(\Nmod) \decal{-r} .
\]
Since $\im \fsh$ is an irreducible component of the support of $\shH^r \bigl(
\FM_A(\Mmod) \bigr)$, it follows that $\Supp \shH^0 \bigl( \FM_B(\Nmod) \bigr) =
\Bsh$, as claimed.
\end{proof}

\subsection{Chern characters}

The purpose of this section is to compute the algebraic Chern character of
$\FM_A(\Mmod)$, for $\Mmod$ a holonomic $\Dmod_A$-module. 

For a smooth algebraic variety $X$, we denote by $\CH(X)$ the algebraic Chow ring of
$X$. To begin with, observe that since $\pi \colon \Ash \to \Ah$ is an affine bundle
in the Zariski topology, the pullback map $\pi^{\ast} \colon \CH(\Ah) \to \CH(\Ash)$
is an isomorphism.

\begin{proposition} \label{prop:Chern}
Let $\Mmod$ be a holonomic $\Dmod_A$-module. Then the algebraic Chern character of
the Fourier-Mukai transform $\FM_A(\Mmod)$ lies in the subring of $\CH(\Ah)$
generated by $\CH_1^1(\Ah) = \Pic^0(\Ah)$.
\end{proposition}

\begin{proof}
Since $\pi \colon E(A) \to \Ah$ is an algebraic vector bundle containing $\Ash =
\lambda^{-1}(1)$, pullback of cycles induces isomorphisms
\[
	\CH(\Ah) \simeq \CH \bigl( E(A) \bigr) \simeq \CH(\Ash).
\]
As in the proof of Proposition~\ref{prop:KW}, choose a good filtration $F_{\bullet}
\Mmod$ and consider the Fourier-Mukai transform $\FMt_A(R_F \Mmod)$ of the associated
Rees module. Its restriction to $\Ash$ is isomorphic to $\FM_A(\Mmod)$, and so it
suffices to show that the Chern character of $\FMt_A(R_F \Mmod)$ is contained in the
subring generated by $\Pic^0(\Ah)$. Since $\lambda^{-1}(0) = \Ah \times H^0(A,
\Omega_A^1)$, we only need to prove this after restricting to $\Ah \times \{\omega\} \subseteq
\lambda^{-1}(0)$, for any choice of $\omega \in H^0(A, \Omega_A^1)$.

By Proposition~\ref{prop:compatible}, the restriction of $\FMt_A(R_F \Mmod)$ to
$\lambda^{-1}(0)$ is isomorphic to 
\begin{equation} \label{eq:FM-temp}
	\derR (p_{23})_{\ast} \Bigl( p_{12}^{\ast} P \tensor p_1^{\ast} \Omega_A^g
			\tensor p_{13}^{\ast} \shC(\Mmod, F) \Bigr).
\end{equation}
Since $\Mmod$ is holonomic, the support of $\shC(\Mmod, F)$ is of pure dimension $g$.
Now choose $\omega \in H^0(A, \Omega_A^1)$ general enough that the restriction of
$\shC(\Mmod, F)$ to $A \times \{\omega\}$ is a coherent sheaf with zero-dimensional
support. The restriction of \eqref{eq:FM-temp} to $\Ah \times \{\omega\}$ is then
the Fourier-Mukai transform of a coherent sheaf on $A$ with zero-dimensional support;
its algebraic Chern character must therefore be contained in the subring of
$\CH(\Ah)$ generated by $A = \Pic^0(\Ah)$.
\end{proof}

\begin{corollary}
Let $\Mmod$ be a holonomic $\Dmod_A$-module. Then all the Chern classes of
$\FM_A(\Mmod)$ are zero in the singular cohomology ring of $\Ash$.
\end{corollary}

\end{document}